\RequirePackage{fix-cm}
\documentclass[smallextended]{svjour3}       
\smartqed  
\usepackage{graphicx}
\usepackage{amsmath}
\usepackage{amsfonts}
\usepackage{amssymb}
\usepackage{comment}
\usepackage{color}
\journalname{}
\def\fam{\eta}
\def\separ{\leftarrow}
\def\sepi{\,|\,}
\def\Fai{\Upsilon}
\def\Cai{\Lambda}
\def\CIP{{\sf C}}
\def\FVP{{\sf F}}
\def\KON{{\sf K}}
\def\POL{{\sf P}}
\def\conv{\mbox{\rm conv}\,}
\def\DAG{\mbox{\sf DAGS}\,(N)}
\def\Id{I}
\def\c{\mbox{\sf c}}
\def\u{\mbox{\sf u}}
\def\calI{{\cal I}}
\def\calP{{\cal P}}
\def\calQ{{\cal Q}}
\def\pa{\mbox{\sl pa}}
\def\ppa{\mbox{\scriptsize\sl pa}}
\def\pr{\mbox{\sl pre}}
\def\dv{\mathbb}
\def\obj{o}
\def\ym{m}
\def\zob{z}
\def\setdags{{\sf S}}

\def\barF{\bar{F}}
\def\Zz{Z}

\begin{document}

\title{Polyhedral aspects of score equivalence in Bayesian network
structure learning
}

\titlerunning{Polyhedral aspects of score equivalence}        

\author{James Cussens \and
        David Haws \and
        Milan Studen\'{y}}

\authorrunning{Cussens, Haws, and Studen\'{y}} 

\institute{J. Cussens \at
              University of York \\
              Deramore Lane, York, YO10 5GE, UK\\
              \email{james.cussens@york.ac.uk}           
              \and
              D. C. Haws \at
              Thomas J. Watson Research Center,\\
              Yorktown Heights, 1101 Kitchwan Road 134, NY, USA\\
              \email{dhaws@us.ibm.com}
              \and
              M. Studen\'{y} \at
              The Institute of Information Theory and Automation of the CAS,\\
              Pod Vod\'{a}renskou v\v{e}\v{z}\'{\i} 4, Prague, 182 08, Czech Republic\\
              \email{studeny@utia.cas.cz}
              }

\date{Received: date / Accepted: date}

\maketitle

\begin{abstract}
This paper deals with faces and facets of the {\em family-variable polytope} and
the {\em characteristic-imset polytope}, which are special polytopes used
in integer linear programming approaches to statistically learn Bayesian network structure.
A common form of linear objectives to be maximized in this area leads to
the concept of {\em score equivalence} (SE), both for linear objectives and for faces of the
family-variable polytope.

We characterize the linear space of SE objectives and establish a \mbox{one-to-one}
correspondence between SE faces of the family-variable polytope, the faces of
the characteristic-imset polytope, and standardized supermodular
functions. The characterization of SE facets in terms of extremality of the
corresponding supermodular function gives an elegant method to verify whether an \mbox{inequality}
is SE-facet-defining for the family-variable polytope.

We also show that when maximizing an SE objective one can eliminate
linear constraints of the family-variable polytope that correspond to non-SE facets.
However, we show that solely considering SE facets is not enough as a counter-example shows;
one has to consider the linear inequality constraints that correspond to facets of the
characteristic-imset polytope despite the fact that they may not define facets in the family-variable mode.

\end{abstract}

\section{Introduction}\label{sec.intro}

The motivation for our paper is statistically learning {\em Bayesian network\/} (BN) structure.
Bayesian networks are popular models used in statistics \cite{Lau96} and probabilistic reasoning
\cite{Pea88}. Acyclic directed graphs, whose nodes correspond to random variables in consideration, are used to
describe the probabilistic conditional independence structures behind the statistical models \cite{Stu05}.

Specifically, our motivation comes from the {\em integer linear programming} (ILP)
approach to the statistical learning task to determine the structural model on basis
of observed data. Nowadays, the most popular is the score-based approach
consisting in maximizing a {\em scoring criterion} $G\mapsto \calQ (G,D)$, where $G$
is an acyclic directed graph, $D$ the observed database and the value $\calQ (G,D)$
says how much the BN structure defined by the graph $G$ explains the occurrence
the database $D$ \cite{Nea04}.

The point of the ILP approach is that the criteria used in practice can be viewed
as (the restriction of) affine functions of suitable vector representatives of
BN structures, typically of acyclic directed graphs. The most common is the
{\em family-variable} vector representation of the graphs suggested independently in
\cite{JSGM10} and \cite{Cus10}. Very good running times have recently been achieved
using this vector representation and the branch-and-cut approach \cite{BC13,GOBNILP}.
The corresponding {\em family-variable polytope}, defined as the convex hull of these
vector representatives, is one of the topics of interest in this paper.

Another ILP approach based on {\em characteristic-imset} vector representation
of BN structures was suggested in \cite{HLS12}; its motivational sources date back to
\cite{Stu05}. Unlike the family-variable vectors, the characteristic imsets uniquely
correspond to BN structures. This ILP approach is also feasible \cite{SH14},
but has not resulted in better running times than those achieved using GOBNILP
software \cite{GOBNILP}. The other polytope we are interested in this paper is
the {\em characteristic-imset polytope}, defined as the convex hull of all
characteristic imsets.

Our paper is devoted to the comparison of the facet-defining inequalities for
the two above-mentioned polytopes, because such inequalities appear to be the most useful
ones in the cutting plane approach to solving ILP problems \cite{Wol98}.
There were some former results on this comparison topic in \cite{SH13}, but the
present paper brings further and deeper findings.
\smallskip

The structure of the paper is as follows. In \S\,\ref{sec.notation} we introduce our
notation and recall basic concepts; elementary facts on polytopes we need later
are gathered in \S\,\ref{sec.formal-def}. Some fundamental observations on
facets of the family-variable polytope, on which our later considerations are based, are
in \S\,\ref{sec.fvp-facets}; some of these facts are also shown using different arguments
in a parallel paper \cite{CJKB15}.

In \S\,\ref{sec.SE-defin} we pinpoint the concept of {\em score equivalence} (SE),
both for linear objectives to be maximized and for faces of the family-variable polytope.
We characterize the linear space of SE objectives in \S\,\ref{sec.SE-char}.
Later, in \S\,\ref{sec.weak-proof}, we establish a one-to-one correspondence
between SE faces of the family-variable polytope and standardized supermodular set functions.
The most beneficial seems to be the characterization of SE facets as those that correspond
to extreme supermodular functions.

Section \S\,\ref{sec.cluster-ineq} deals with well-known (generalized)
{\em cluster inequalities} applied dominantly in contemporary ILP approaches
to BN structure learning. We find the corresponding supermodular functions and show
they are extreme. This gives a simple proof that the generalized cluster inequalities are facet-defining
for the family-variable polytope; note that another proof of this fact,
based on different arguments, will appear in \cite{CJKB15}.
We also interpret the generalized cluster inequalities in terms of connected uniform matroids.

Another one-to-one correspondence between SE faces of the family-variable polytope
and faces of the characteristic-imset polytope is established in \S\,\ref{sec.fac-cores};
to illustrate this correspondence we derive the form of cluster inequalities
in the characteristic-imset mode. A few simple examples are given in \S\,\ref{sec.examples}.

Further important observation of ours are in \S\,\ref{sec.revised-conj}:
when maximizing an SE objective, one actually need not apply the linear
facet-defining constraints on the family-variable polytope that are not SE. On the other hand, considering
only SE facets is not enough as a later counter-example in \S\,\ref{sec.five-var} shows.
Thus, we also reveal the hidden importance of the linear constraints that correspond to facets of the characteristic-imset polytope  in \S\,\ref{sec.revised-conj} .

The appendix contains the proof of an auxiliary combinatorial identity (\S\,\ref{sec.comb-iden}),
the catalogue of SE facets in case of four BN variables (\S\,\ref{sec.SE-four}) and the
catalogue of remaining facets of the characteristic-imset polytope in case of four BN variables
(\S\,\ref{sec.CIP-four}).

\section{Notation and basic concepts}\label{sec.notation}
Let $N$ be a finite non-empty set of BN variables; $n:= |N|<\infty$,
consider the non-trivial case $2\leq n$. Let $\DAG$ denote the
collection of acyclic directed graphs over $N$, that is, such graphs
having $N$ as the set of nodes.  An example of such a graph is the
{\em empty graph\/}, which is a graph over $N$ without adjacencies.
By a {\em full graph\/} we will mean any acyclic directed graph over
$N$ in which every pair of distinct nodes is adjacent. Given
$G\in\DAG$ and $a\in N$, the symbol $\pa_{G}(a):= \{b\in N:\ a\to b
~~\mbox{in $G$}\}$ will denote the {\em parent set\/} of the node
$a$.  A well-known equivalent definition of acyclicity of a directed
graph $G$ over $N$ is the existence of a total order $a_{1},\ldots
,a_{n}$ of nodes in $N$ such that, for every $i=1,\ldots ,n$,
$\pa_{G}(a_{i})\subseteq\{ a_{1},\ldots ,a_{i-1}\}$; we say then that the
order and the graph are {\em consonant}.  An {\em immorality\/} in $G$
is an induced subgraph of $G$ of the form $a\to c\leftarrow b$, where
the nodes $a$ and $b$ are not adjacent in $G$.  \smallskip

The symbol $G\sim H$ for $G,H\in\DAG$ will mean that the graphs $G$ and $H$
are {\em Markov equivalent}, that is, in graphical terms, they have the same adjacencies and immoralities;
for references see \cite[p.\,60]{Lau96} or \cite[p.\,48-49]{Stu05}.
An example of a Markov equivalence class is the set of full graphs over $N$.

A node $a$ together with its parent set $B$ will be called a
\emph{family}. Note that any directed graph is determined by its $N$ families.
Throughout the paper, the index set of family-variable vectors will be
$$
\Fai := \{\, (a\sepi B)\, :\
a\in N ~~\& ~~ \emptyset\neq B\subseteq N\setminus\{ a\}\,\}\,.
$$
Note that families with empty parent sets are not included.

Given $b\in  N$ and $\Zz\subseteq N\setminus\{ b\}$
the symbol $\Id_{b\separ\Zz}$ will be used to denote the identifier of this pair, that is,
an element of ${\dv R}^{\Fai}$ given by
$$
\Id_{b\separ\Zz}(a\sepi B) =
\left\{
\begin{array}{ll}
\,1 & ~~\mbox{if $a=b$ and $B=\Zz$,}\\
\,0 & ~~\mbox{otherwise,}
\end{array}
\right.
\quad \mbox{for any $(a\sepi B)\in\Fai$.}
$$
In case $\Zz =\emptyset$, for any $b\in N$, $\Id_{b\separ\Zz}=\Id_{b\separ \emptyset}$ is the zero vector.
The symbol $\fam_{G}$ will be used to denote the family-variable vector encoding $G\in\DAG$:
$$
\fam_{G}(a\sepi B) =
\left\{
\begin{array}{ll}
\,1 & ~~\mbox{if $B=\pa_{G}(a)$,}\\
\,0 & ~~\mbox{otherwise,}
\end{array}
\right.
\quad \mbox{for $(a\sepi B)\in\Fai$.}
$$
The {\em family-variable polytope} can be
defined as the convex hull of the set of all possible DAG-codes over $N$:
$$
\FVP := \conv(\{\, \fam_{G}\in {\dv R}^{\Fai}\, :\  G\in \DAG\,\})\,.
$$
Clearly, the dimension of\/ $\FVP$, defined as the dimension of its linear hull,
is $\dim (\FVP)=|\Fai |=n\cdot (2^{n-1}-1)$. It is easy to see that none of the
DAG-codes is a non-trivial convex combination of the others. In particular, the set of
vertices (= extreme points) of $\FVP$ is just the set of DAG-codes.
\smallskip

Given two vectors $v,w\in {\dv R}^{\Gamma}$, where $\Gamma$ is a non-empty finite index set, say $\Gamma =\Fai$,
their scalar product will be denoted by $\langle v,w\rangle_{\Gamma}$, or just by $\langle v,w\rangle$
if there is no danger of confusion. We also consider alternative index sets.
\smallskip

Specifically, the {\em characteristic imset\/} of $G\in\DAG$, introduced in \cite{HLS12} and denoted below by $\c_{G}$,
is an element of ${\dv R}^{\Cai}$ with
$$
\Cai := \{\, S\subseteq N\, :\ |S|\geq 2\,\}\,.
$$
Recall from \cite[\S\,3.3.2]{SH13} and \cite[\S\,2]{BC13} that $\c_{G}$ is
a many-to-one linear function of $\fam_{G}$; the transformation is $\fam\mapsto \c_{\fam}$,
where
\begin{equation}
\c_{\fam}(S) = \sum_{a\in S}\, \ \sum_{B\,:\,S\setminus\{ a\}\subseteq B\subseteq N\setminus\{ a\}}
\fam (a\sepi B)\qquad \mbox{for any $S\subseteq N$, $|S|\geq 2$.}
\label{eq.eta-to-char}
\end{equation}
A further fundamental observation is that $G\sim H$ for $G,H\in\DAG$ iff
$\c_{G}=\c_{H}$; see \cite[\S\,3]{HLS12} for more detailed justification.
The {\em characteristic-imset polytope\/} is defined as follows:
$$
\CIP := \conv(\{\, \c_{G}\in {\dv R}^{\Cai}\, :\  G\in \DAG\,\})\,.
$$
One can show that $\dim (\CIP)=|\Cai |= 2^{n}-n-1$. Of course, $\CIP$ is the image of $\FVP$ by the
linear map \eqref{eq.eta-to-char}.
\smallskip

Moreover, the power set $\calP (N) := \{ A\,:\; A\subseteq N\}$ will serve as an index set
for vectors, used as auxiliary tools in a later proof in \S\,\ref{sec.fac-cores}.
Given $A\subseteq N$, let us denote its indicator vector by
$$
\delta_{A}(S) = \left\{ \begin{array}{ll}
1 & ~~\mbox{\rm if}~~ S=A\,,\\
0 & ~~\mbox{\rm if}~~ S\subseteq N,\; S\neq A\,,
\end{array} \right.
$$
and define the {\em standard imset\/} for $G\in\DAG$ as an element of ${\dv R}^{\calP (N)}$:
\begin{equation}
\u_{G} := \delta_{N} - \delta_{\emptyset} + \sum_{a\in N}\,\left\{\,\delta_{\ppa_{G}(a)} -
\delta_{\{ a\}\cup\ppa_{G}(a)}\,\right\}.
\label{eq.def-stand}
\end{equation}
Recall from \cite[\S\,3.3]{SH13} that $\c_{G}$ is a one-to-one affine function of $\u_{G}$, specifically
\begin{equation}
\c_{G} (T) = 1-  \sum_{S:\, T\subseteq S\subseteq N} \u_{G}(S) \qquad \mbox{for $T\subseteq N$, $|T|\geq 2$}.
\label{eq.stan-to-char}
\end{equation}
In particular, the combination of a former characterization \cite[Theorem 4]{SVH10} of
the vertices of the standard-imset polytope with \eqref{eq.stan-to-char} implies
that the set of vertices (= extreme points) of the characteristic-imset polytope $\CIP$ is
just the set of characteristic imsets $\c_{G}$ for $G\in\DAG$. In other words, no
characteristic imset is a non-trivial convex combination of the others.

\section{Elementary facts on facets and some conventions}\label{sec.formal-def}
Recall the basic concept of a face/facet of a polytope.

 \begin{definition}[dimension, face, facet]\rm\label{def.face} ~~\\
 Let $\POL$ be a polytope in ${\dv R}^{\Gamma}$, where $\Gamma\neq\emptyset$\/ is finite.
 Its {\em dimension} is defined as the dimension of its affine hull, which is a translate
 of a linear subspace of ${\dv R}^{\Gamma}$.
 A set $F\subseteq\POL$ is called a {\em face\/} of $\POL$ if there exists a vector
 $\obj\in{\dv R}^{\Gamma}$ and a constant $u\in {\dv R}$ such that
 \begin{itemize}
 \item $\POL\subseteq \{\, v\in {\dv R}^{\Gamma}\, :\ \langle\obj , v\rangle\leq u\,\}$,  and
 \item $F = \{\, v\in\POL\, :\ \langle\obj , v\rangle =u\,\}$\,.
 \end{itemize}
 We say then that the face $F$ is defined by the inequality $\langle\obj , v\rangle\leq u$.
 Every face of a polytope is a (possibly empty) polytope, as well;
 thus, its dimension is defined. A {\em facet} of $\POL$ is a face of dimension $\dim (\POL )-1$.

 The function $v\in{\dv R}^{\Gamma}\mapsto \langle\obj , v\rangle$, where $\obj\in{\dv R}^{\Gamma}$, is
 typically a linear objective to be maximized by a linear program;  with a small abuse of terminology we will call
 $\obj\in {\dv R}^{\Gamma}$ an {\em objective}.
 \end{definition}

 Note that the dimension of a face is one less than the maximum number of affinely
 independent vectors in the face. An alternative equivalent definition of a facet
 is that it is a sub-maximal face with respect to inclusion.

 \begin{lemma}\label{lem.facet=maxim}\rm
 Given a polytope $\POL$ in ${\dv R}^{\Gamma}$, $0<|\Gamma |<\infty$, a face $F\subset\POL$
 is a facet of $\POL$ iff the only face $F^{\prime}$ of $\POL$ with $F\subset F^{\prime}$ is $F^{\prime}=\POL$ itself.
 \end{lemma}

 \begin{proof}
 The sufficiency follows from the fact that, for every pair of faces $F_{1}\subset F_{2}$ of
 $\POL$ with $\dim (F_{1})<d<\dim (F_{2})$, a face $F_{3}$ of $\POL$ exists with $F_{1}\subset F_{3}\subset F_{2}$ and
 $\dim (F_{3})=d$; see, for example, \cite[Corollary 9.7]{Bro83}.
 For the necessity realize that, if $F_{1}\subset F_{2}$ are faces of $\POL$ then $\dim (F_{1})<\dim (F_{2})$;
 see \cite[Corollary 5.5]{Bro83}.
 \end{proof}

 The consequence is an auxiliary observation, applied later in the paper.

 \begin{corollary}\label{cor.facet=maxim}\rm
 Let $\POL\subseteq {\dv R}^{\Gamma}$, $0<|\Gamma |<\infty$ be a polytope and
 let $\langle\obj_{1},v\rangle\leq u_{1}$ and
 $\langle\obj_{2},v\rangle\leq u_{2}$ be valid inequalities
 for $v\in\POL$ such that
 \begin{equation}
 \exists\, w_{1}\in\POL\, :\ \langle\obj_{1},w_{1}\rangle <u_{1} ~\&~ \langle\obj_{2},w_{1}\rangle =u_{2} ~~\mbox{and}~~
 \exists\, w_{2}\in\POL\, :\ \langle\obj_{2},w_{2}\rangle <u_{2}.
 \label{eq.non-facet}
 \end{equation}
 Then no combination of these inequalities $\langle\alpha\cdot\obj_{1}+\beta\cdot\obj_{2},v\rangle\leq \alpha\cdot u_{1}+\beta\cdot u_{2}$
 with $\alpha ,\beta >0$ is a facet-defining inequality for $\POL$.
 \end{corollary}

 \begin{proof}
 Let $F_{1}$, $F_{2}$ and $F$ be the faces of\/ $\POL$ defined by  inequalities
 $\langle\obj_{1},v\rangle\leq u_{1}$, $\langle\obj_{2},v\rangle\leq u_{2}$
 and their combination $\langle\alpha\cdot\obj_{1}+\beta\cdot\obj_{2},v\rangle\leq \alpha\cdot u_{1}+\beta\cdot u_{2}$, respectively. Given $v\in F$ one has
 $$
 \alpha\cdot \{\underbrace{\langle\obj_{1},v\rangle -u_{1}}_{\leq 0}\}
 + \beta\cdot \{\underbrace{\langle\obj_{2},v\rangle -u_{2}}_{\leq 0}\} =0,
 $$
 which implies that the expressions in braces must vanish. In other words, $F\subseteq F_{1}\cap F_{2}$.
 Assume for a contradiction that $F$ is a facet. By Lemma \ref{lem.facet=maxim} observe that either $F_{1}=\POL$ or $F_{1}=F$; the same
 for $F_{2}$. Since \eqref{eq.non-facet} implies $w_{1}\in\POL\setminus F_{1}$ and $w_{2}\in\POL\setminus F_{2}$, one necessarily
 has $F_{1}=F=F_{2}$. However, this contradicts the existence of $w_{1}\in F_{2}\setminus F_{1}$ assumed in \eqref{eq.non-facet}.
 \end{proof}

 In this paper we mainly deal with the family-variable polytope $\FVP$.
 Every face of\/ $\FVP$ can be identified with a set of acyclic directed graphs. Specifically:
 $$
 F\subseteq\FVP ~~\mbox{a face of\/ $\FVP$} ~~\longleftrightarrow ~~
 \setdags =\{ G\in\DAG :\ \fam_{G}\in F\}\,.
 $$
 This correspondence preserves inclusion, that is, $F_{1}\subseteq F_{2}$ for faces of\/ $\FVP$
 iff $\setdags_{1}\subseteq\setdags_{2}$ for the corresponding sets of graphs $\setdags_{i}\subseteq\DAG$.
 The identification is possible owing to a basic fact from the theory of polytopes
 that every face $F$ of a polytope $\POL$ is the convex hull of the set of vertices of $\POL$
 which belong to $F$, see \cite[Lemma VI.1.1]{Bar02} or \cite[Proposition 2.3(i)]{Zie95}.
 Since the vertices of\/ $\FVP$ are just the DAG-codes $\fam_{G}$, where $G\in\DAG$,
 every face $F$ of\/ $\FVP$ can be identified with a subset of $\DAG$.
 This leads to the following convention.

 \begin{definition}[a set of graphs interpreted as a face]\rm\label{def.face-DAG} ~~\\
 We will call a set $\setdags\subseteq\DAG$  a {\em face} (of the family-variable polytope $\FVP$) if
 $\conv (\{\,\fam_{G}\in {\dv R}^{\Fai} :\ G\in\setdags\})$ is a face of\/ $\FVP$.
 Analogously, $\setdags\subseteq\DAG$ will be called a {\em facet} (of $\FVP$) if
 $\conv (\{\,\fam_{G}\in {\dv R}^{\Fai} :\ G\in\setdags\})$ is a facet of\/ $\FVP$.
 \end{definition}

 A direct method to show that a face $\setdags\subseteq\DAG$ is a facet is to show
 that the respective geometric face $F$ has the dimension $\dim (\FVP )-1$, which means, to find
 $\dim (\FVP )$ affinely independent vectors in $F$.
 Since the vertices of $F$ are just the family-variable vectors for $G\in\setdags$, the task, more or less,
 reduces to the question of finding a subset $\setdags^{\prime}\subseteq\setdags$ of cardinality
 $|\Fai |=n\cdot (2^{n-1}-1)$ such that the vectors $\{\,\fam_{G}\in {\dv R}^{\Fai} :\ G\in\setdags^{\prime}\}$
 are affinely independent.
 \smallskip

 We accept a {\em standardization convention\/} that valid inequalities for vectors $\fam\in\FVP$ in
 the family-variable polytope will be written in the upper-bound form:
 \begin{equation}
 \langle\obj ,\fam\rangle\leq u\quad
 \mbox{where $\obj\in {\dv R}^{\Fai}$ is an objective and $u\in {\dv R}$ an upper bound.}
 \label{eq.up-bou-ineq}
 \end{equation}
 Note that any lower-bound inequality $\langle\obj ^{\prime},\fam\rangle\geq l$ can be replaced by $\langle\obj ,\fam\rangle\leq u$
 where $\obj =-\obj^{\prime}$ and $u=-l$. Since $\FVP$ is a rational polytope, its facets are defined by inequalities with
 rational coefficients, that is, by \eqref{eq.up-bou-ineq} with $\obj\in {\dv Q}^{\Fai}$.
 By multiplying it by a suitable positive factor one can get (unique) integer vector objective $\obj\in {\dv Z}^{\Fai}$ whose components have no common
 prime divisor. Since the vertices of\/ $\FVP$ are zero-one vectors, the tight upper bound
 in \eqref{eq.up-bou-ineq} must be then an integer as well: $u\in {\dv Z}$.
 \smallskip

 Moreover, a couple of special {\em extension conventions\/} for vectors
 in ${\dv R}^{\Fai}$ and ${\dv R}^{\Cai}$ will be accepted to simplify some later formulas:
 \begin{itemize}
 \item for every {\em objective} $\obj\in {\dv R}^{\Fai}$, assume $\obj (b\sepi\emptyset)=0$ for any $b\in N$,
 \item for any $\ym\in {\dv R}^{\Cai}$, put $\ym (S)=0$ for $S\subseteq N$, $|S|\leq 1$.
 \end{itemize}

\section{Observations on facets of the family-variable polytope}\label{sec.fvp-facets}
In this section, we present a few general facts concerning faces and facets of\/ $\FVP$ and
describe explicitly those facets which contain the empty graph.
Note that some of these basic observations are also mentioned and used in a parallel paper \cite{CJKB15}.
We keep the standardization convention from \S\,\ref{sec.formal-def}.
A basic division of facet-defining inequalities is on the basis of the upper bound value $u$.

\begin{lemma}\label{lem.increase-facet}\rm
Assume that \eqref{eq.up-bou-ineq}, that is, the inequality
$\langle\obj ,\fam\rangle\leq u$ with $\obj\in {\dv R}^{\Fai}$ and $u\in {\dv R}$,
is a valid inequality for all $\fam\in\FVP$.
Then $u\geq 0$.
\begin{itemize}
\item[(i)] One has $u=0$ iff the corresponding face of\/ $\FVP$ contains the empty graph.\smallskip
\item[(ii)] If $u=0$ then the objective coefficients are non-positive:
$$
\obj (a\sepi B)\leq 0\qquad \mbox{for each ~$(a\sepi B)\in\Fai$}.
$$
\item[(iii)] The facet-defining inequalities tight at the empty graph are
just
\begin{equation}
-\fam (a\sepi B)\leq 0\qquad \mbox{for each~ $(a\sepi B)\in\Fai$}.
\label{eq.facet-empty}
\end{equation}
\item[(iv)] If \eqref{eq.up-bou-ineq} is a facet-defining inequality
for $\FVP$ with $u>0$ then the objective coefficients are non-negative
and increasing in the following sense:
$$
\obj (a\sepi B)\geq \obj (a\sepi A)\geq 0\quad \mbox{whenever $a\in N$ and~ $\emptyset\neq A\subseteq B\subseteq N\setminus\{ a\}$}.
$$
\end{itemize}
\end{lemma}

Note that an alternative proof of Lemma~\ref{lem.increase-facet}(iv) is
in \cite[Proposition\,5]{CJKB15}.

\begin{proof}
The zero vector in ${\dv R}^{\Fai}$ is the code for the empty graph and, therefore,
belongs to $\FVP$. The substitution of $\fam =0$ into \eqref{eq.up-bou-ineq} gives $0\leq u$.
It is clear that the inequality is tight for $\fam =0$ iff $u=0$, which gives (i).

As concerns (ii), assume for a contradiction that $(a\sepi B)\in\Fai$ such that $\obj (a\sepi B)>0$ exists in
\eqref{eq.up-bou-ineq} with $u=0$. Consider $G\in\DAG$ with $\fam_{G} =\Id_{a\separ B}$.
Then $\langle\obj ,\fam_{G}\rangle = \obj (a\sepi B)>0=u$ contradicts the validity of
\eqref{eq.up-bou-ineq}.

As concerns (iii), an elementary fact is that, for every $(b\sepi D)\in\Fai$,
all the inequalities in \eqref{eq.facet-empty} with $(a\sepi B)\neq (b\sepi D)$ are tight for
the family-variable vector $\fam =\Id_{b\separ D}\in\FVP$ but not the inequality corresponding to $(b\sepi D)$.
This allows us to observe that any inequality in \eqref{eq.facet-empty} is facet-defining for $\FVP$.
Indeed, any such inequality is valid for $\FVP$ and, having fixed $(a\sepi B)\in\Fai$,
the respective inequality $-\fam (a\sepi B)\leq 0$ is tight for $|\Fai |$  affinely independent vectors,
namely the zero vector in ${\dv R}^{\Fai}$ and vectors $\Id_{b\separ D}$ for $(b\sepi D)\neq (a\sepi B)$.
The second step is to show that every facet $F$ of\/ $\FVP$ containing the empty graph is defined by \eqref{eq.facet-empty}. Former observations (i) and (ii) imply that the facet-defining inequality for $F$ must have the form
$\langle\obj ,\fam\rangle\leq 0$ with $\obj\in (-\infty,0]^{\Fai}$. Thus, the inequality is a conic combination of
those from \eqref{eq.facet-empty}. Since $F$ is assumed to be a facet, Corollary \ref{cor.facet=maxim} can be used
to show that at most one coefficient in the combination is non-zero. Indeed, if two coefficients $\obj (a\sepi B)$ and $\obj (b\sepi D)$
are non-zero, the above elementary fact implies for the inequalities $\obj (a\sepi B)\cdot\fam (a\sepi B)\leq 0$ and
$\sum_{(c\sepi E)\neq (a\sepi B)} \obj (c\sepi E)\cdot\fam (c\sepi E)\leq 0$ that the condition \eqref{eq.non-facet} from Corollary \ref{cor.facet=maxim}
is fulfilled with $w_{1}=\Id_{a\separ B}$ and $w_{2}=\Id_{b\separ D}$. On the other hand, at least one
coefficient must be non-zero, since otherwise $F=\FVP$. Therefore, the facet $F$ must be defined by one
of the inequalities in \eqref{eq.facet-empty}.

As concerns (iv), owing to the extension convention from \S\,\ref{sec.formal-def}, the statement means
$\obj (a\sepi B)\geq \obj (a\sepi A)$ for $a\in N$ and $A\subseteq B\subseteq N\setminus\{ a\}$.
Assume for a contradiction that $a\in N$ and $A\subset B\subseteq N\setminus\{ a\}$ exist with
$\obj (a\sepi B)<\obj (a\sepi A)$ and define $\tilde{\obj}\in {\dv R}^{\Fai}$ in the following way:
$$
\tilde{\obj }(b\sepi D):=
\left\{
\begin{array}{ll}
\obj (b\sepi D) & ~~~~\mbox{for $(b\sepi D)\in\Fai$, $(b\sepi D)\neq (a\sepi B)$,}\\
\obj (a\sepi A) & ~~~~\mbox{for $(b\sepi D)= (a\sepi B)$}.
\end{array}
\right.
$$
The next observation is that $\langle \tilde{\obj },\fam\rangle\leq u$ is a valid inequality for $\FVP$.
Specifically, given $G\in\DAG$, construct $\tilde{G}\in\DAG$ such that $\fam_{\tilde{G}}(a\sepi B)=0$
and $\langle \tilde{\obj },\fam_{G}\rangle=\langle \tilde{\obj },\fam_{\tilde{G}}\rangle$.
Indeed, if $\pa_{G}(a)\neq B$ then simply $\tilde{G}:=G$, otherwise put
$\pa_{\tilde{G}}(a)=A$ and $\pa_{\tilde{G}}(b)=\pa_{G}(b)$ for $b\in N\setminus\{ a\}$,
which gives
$$
\langle \tilde{\obj },\fam_{G}\rangle- \langle \tilde{\obj },\fam_{\tilde{G}}\rangle =
\tilde{\obj }(a\sepi B)-\tilde{\obj }(a\sepi A)=\obj (a\sepi A)-\obj (a\sepi A)=0\,.
$$
The definition of $\tilde{\obj }$ implies $\langle \tilde{\obj },\fam_{\tilde{G}}\rangle -\langle\obj ,\fam_{\tilde{G}}\rangle
= \{\tilde{\obj }(a\sepi B)-\obj (a\sepi B)\}\cdot\fam_{\tilde{G}}(a\sepi B)=0$. Because \eqref{eq.up-bou-ineq}
is valid for $\fam_{\tilde{G}}$ one can observe
$$
\langle \tilde{\obj },\fam_{G}\rangle= \langle \tilde{\obj },\fam_{\tilde{G}}\rangle =
\langle\obj ,\fam_{\tilde{G}}\rangle\leq u\,,
\quad \mbox{which was desired.}
$$
Thus, \eqref{eq.up-bou-ineq} is the sum of the valid inequality
$\langle \tilde{\obj },\fam\rangle\leq u$ with a positive multiple of the valid inequality
$-\fam (a\sepi B)\leq 0$, namely by $\beta :=\obj (a\sepi A)-\obj (a\sepi B)>0$.
The condition \eqref{eq.non-facet} from Corollary \ref{cor.facet=maxim}
is fulfilled with $w_{1}=0$ and $w_{2}=\Id_{a\separ B}$, which implies a contradictory
conclusion that \eqref{eq.up-bou-ineq} is not facet-defining.
\end{proof}

This implies the following observation.

\begin{corollary}\label{cor.increase-facet}\rm
Let $\setdags$ be a facet of $\FVP$ in the sense of Definition
\ref{def.face-DAG} which does not contain the empty graph.
Then $\setdags$ is closed under super-graphs in the sense:
$$
\mbox{if $G\in\setdags$ is a subgraph of $H\in\DAG$ ~~~~then $H\in\setdags$}.
$$
Moreover, for every $(a\sepi B)\in\Fai$, there exists $G\in\setdags$ with $\pa_{G}(a)=B$.
\end{corollary}

The second statement in Corollary \ref{cor.increase-facet} is also derived
in \cite[Proposition\,4]{CJKB15} using slightly different arguments.

\begin{proof}
It is enough to verify the first claim in the case $H$ differs from $G$ just in just one parent
set, that is, in case $a\in N$ exists with $A=\pa_{G}(a)\subset\pa_{H}(a)=B$ and $\pa_{H}(b)=\pa_{G}(b)$ for
$b\in N\setminus\{ a\}$. By Lemma \ref{lem.increase-facet}(i), we know that $\setdags$ is given by the inequality
\eqref{eq.up-bou-ineq} with $u>0$. Thus, by Lemma \ref{lem.increase-facet}(iv), one can write
$\langle\obj ,\fam_{H}\rangle - \langle\obj ,\fam_{G}\rangle = \obj (a\sepi B)-\obj (a\sepi A)\geq 0$.
Assuming $G\in\setdags$, the inequality \eqref{eq.up-bou-ineq} is tight for $\fam_{G}$ and one has
$$
u=\langle\obj ,\fam_{G}\rangle \leq \langle\obj ,\fam_{H}\rangle \leq u
\qquad \mbox{because \eqref{eq.up-bou-ineq} is valid for $\fam_{H}$.}
$$
Hence, $\langle\obj ,\fam_{H}\rangle =u$, that is, \eqref{eq.up-bou-ineq}
is tight for $\fam_{H}$, saying that $H\in\setdags$.

As concern the second claim assume for a contradiction that
$(a\sepi B)\in\Fai$ exists with $\pa_{G}(a)\neq B$ for any $G\in\setdags$.
That means, $\setdags$ is contained in the face defined by $-\fam (a\sepi B)\leq 0$.
Since $\conv (\{\,\fam_{G}\in {\dv R}^{\Fai} :\ G\in\setdags\})$ is a facet of\/ $\FVP$, by Lemma \ref{lem.facet=maxim}, observe that it coincides with the face
defined by $-\fam (a\sepi B)\leq 0$. This implies a contradictory conclusion
that $\setdags$ contains the empty graph.
\end{proof}

An obvious modification of natural convexity constraints gives the following valid inequalities for
the family-variable polytope:
\begin{equation}
\sum_{B\,:\,\emptyset\neq B\subseteq N\setminus\{ a\}} \fam (a\sepi B)\leq 1\qquad
\mbox{for any $a\in N$}.
\label{eq.mod-conv}
\end{equation}
Except for a degenerate case $n=2$, these inequalities are facet-defining; see also
\cite[Proposition\,3]{CJKB15}.

\begin{lemma}\label{lem.mod-convex-facet}\rm
If $n\geq 3$ then, for every $a\in N$, \eqref{eq.mod-conv} defines a facet of\/ $\FVP$.
\end{lemma}

\begin{proof}
We find $|\Fai |$ affinely independent vectors on the face. Specifically,
for $\emptyset\neq B\subseteq N\setminus\{ a\}$ put $\fam_{(a\sepi B)}=\Id_{a\separ B}$,
while for $b\in N$, $b\neq a$ and $(b\sepi D)\in\Fai$ put
$\fam_{(b\sepi D)}=\Id_{a\separ N\setminus\{ a,b\}}+\Id_{b\separ D}$.
These vectors linearly generate ${\dv R}^{\Fai}$. Hence, they are linearly independent,
and, therefore, affinely independent.
\end{proof}

\section{Score equivalence concept}\label{sec.SE-defin}
The score-based approach to structural learning Bayesian networks consists
in maximization of a function $G\in\DAG\mapsto\calQ (G,D)$, where $D$ is the database
of observed values and $\calQ$ a suitable {\em quality criterion\/}, also called a {\em scoring criterion} \cite[p.\ 437]{Nea04},
which evaluates how the graph $G$ fits the database. The criteria used in practice
turn out to be affine functions of the family-variable vector,
that is, $\calQ (G,D)=k+\langle\obj ,\fam_{G}\rangle_{\Fai}$ with
$k\in {\dv R}$ and $\obj\in {\dv R}^{\Fai}$ encoding both $D$ and $\calQ$.
Thus, the learning task turns into an LP problem to maximize a linear function over the
vertices of the family-variable polytope $\FVP$.
\smallskip

Since the goal is typically to learn the structure, described by a Markov equivalence class of graphs,
most of criteria used in practice do not distinguish between Markov equivalent graphs, that is, one has
$$
\calQ (G,D)=\calQ (H,D)\quad \mbox{whenever $G$ and $H$ are Markov equivalent.}
$$
In the machine learning community, quality criteria satisfying the above condition are called {\em score equivalent\/} \cite{Bou95,Chi02}.
This motivates the following terminology.

\begin{definition}[score equivalent objective]\rm\label{def.SEO} ~~\\
We say that a vector $\obj\in{\dv R}^{\Fai}$ is a {\em score equivalent objective\/}
(abbreviated below as {\em an SE objective}) if it satisfies
\begin{equation}
\forall\, G,H\in\DAG\quad G\sim H ~\Rightarrow ~
\langle\obj ,\fam_{G}\rangle = \langle\obj ,\fam_{H}\rangle\,.
\label{eq.SEO}
\end{equation}
Clearly, the set of SE objectives is a linear subspace of ${\dv R}^{\Fai}$.
\end{definition}

The faces and facets of\/ $\FVP$ are defined in terms of normal vectors, which leads to the following concept.

\begin{definition}[SE face/facet, closed under Markov equivalence]\rm\label{def.s-SEF} ~~\\
We will name a face $F$ of\/ $\FVP$ {\em score equivalent\/} (SE) if there exists
an SE objective $\obj\in{\dv R}^{\Fai}$ and a constant $u\in {\dv R}$ such that
two conditions from Definition \ref{def.face} hold for $\POL =\FVP$.
By an {\em SE facet} is meant a facet of\/ $\FVP$ which is an SE face.

A related concept is the next one: a set $\setdags\subseteq\DAG$ of acyclic directed graphs is {\em closed under Markov equivalence\/} if
\begin{equation}
\forall\, G,H\in\DAG\quad G\sim H \qquad
G\in\setdags ~~~\Rightarrow ~~H\in\setdags .
\label{eq.weakSE}
\end{equation}
\end{definition}
\smallskip

\begin{remark}\rm ~
Note that an objective determining a face is not uniquely determined.
Only in the case of a facet (of a full-dimensional polytope), is it unique up to a positive multiple. Therefore, one has to be careful
when testing score equivalence of a face $F$ which is not a facet, because one of the face-defining
objectives for $F$ could be SE and another objective for $F$ need not be.
Our definition requires the {\em existence\/} of at least one SE objective defining the face.
\end{remark}

The following observation is straightforward.

\begin{lemma}\label{lem.SE>WE}\rm
A set of graphs on an SE face is closed under Markov equivalence.
\end{lemma}

\begin{proof}
Given an SE objective $\obj$ with $F =\{\,\fam\in\FVP\, :\ \langle\obj ,\fam\rangle =u\}$
for some $u\in {\dv R}$ and $G\in\DAG$ with $\langle\obj ,\fam_{G}\rangle =u$, \eqref{eq.SEO}
implies for $H\sim G$ that $\langle\obj ,\fam_{H}\rangle =u$.
\end{proof}

An open question is whether the converse is true.

\begin{conjecture}\label{conj.SE=WEface} ~~\\[0.2ex]
Every face $\setdags\subseteq\DAG$ of\/ $\FVP$ closed under Markov equivalence is an SE face.
\end{conjecture}

We managed to confirm the conjecture for facets; see Theorem \ref{thm.WE=SEfacet} in \S\,\ref{sec.weak-proof}.
The arguments there are slightly special and do not apply to general faces.
However, we were able to verify Conjecture \ref{conj.SE=WEface} for $n=|N|=3$ by an exhaustive analysis.
By means of a computer, we verified for $n=4$ that every inclusion-submaximal face among those
closed under Markov equivalence is already an SE face. Our computational attempts to find a
counter-example for $n=5$ have not been successful.

\section{SE objectives characterization}\label{sec.SE-char}

Recall that to present the characterization of the linear space of SE objectives in an elegant way we use
the extension conventions from \S\,\ref{sec.formal-def}.

\begin{lemma}\rm\label{lem.SEF}
A vector $\obj\in{\dv R}^{\Fai}$ is an SE objective if and only if
either of the following two conditions holds. The two conditions are
equivalent: the first holds if and only if the second does.
\begin{itemize}
\item[(a)] For any $\Zz\subseteq N$ and $a,b\in N\setminus\Zz$, $a\neq b$ one has
\begin{equation}
\obj (b\sepi\{ a\}\cup\Zz ) + \obj (a\sepi\Zz ) =
\obj (a\sepi \{ b\}\cup\Zz ) +\obj (b\sepi\Zz )\,.
\label{eq.SE-dual}
\end{equation}
\item[(b)] There exists $\ym\in{\dv R}^{\Cai}$ such that
\begin{equation}
\obj (a\sepi B) = \ym (\{ a\}\cup B)-\ym (B)\qquad
\mbox{for any $a\in N$, $B\subseteq N\setminus\{ a\}$.}
\label{eq.SEF}
\end{equation}
\end{itemize}
In particular, the dimension of the linear subspace of SE objectives is 
$2^{n}-n-1$.
\end{lemma}

\begin{proof}
The condition \eqref{eq.SEO} for $\obj\in{\dv R}^{\Fai}$ means
$\langle\obj ,\fam_{G}-\fam_{H}\rangle =0$ if $G,H\in\DAG$ are such that $G\sim H$.
A well-known transformational characterization of Markov equivalence
\cite[Theorem 2]{Chi95} says that $G\sim H$ if and only if there exists a sequence $G=G_{1},\ldots , G_{m}=H$, $m\geq 1$
in $\DAG$ such that, for $i=1,\ldots ,m-1$, the graph $G_{i+1}$ is obtained from $G_{i}$
by ``covered arc reversal". This means that $G_{i}$ has an arrow $a\to b$ with
$\pa_{G_{i}} (b)=\{ a\}\cup\pa_{G_{i}} (a)$ and $G_{i+1}$ is obtained from $G_{i}$ by
replacing $a\to b$ in $G_{i}$ by $b\to a$ in $G_{i+1}$; the remaining arrows are unchanged.
In particular, $G_{i}\sim G_{i+1}$ and, provided $\Zz =\pa_{G_{i}} (a)$ one has
$$
\fam_{G_{i}}-\fam_{G_{i+1}} = \Id_{b\separ \{ a\}\cup\Zz} + \Id_{a\separ\Zz}
-\Id_{a\separ \{ b\}\cup\Zz} - \Id_{b\separ\Zz}.
$$
Hence, we easily derive that \eqref{eq.SEO} holds for $\obj\in{\dv R}^{\Fai}$ iff
\eqref{eq.SE-dual} holds.

It remains to show that \eqref{eq.SE-dual} is equivalent to the existence of
$\ym\in{\dv R}^{\Cai}$ such that $\obj $ is given by \eqref{eq.SEF}. The sufficiency
of \eqref{eq.SEF} is easy: then both LHS and RHS in \eqref{eq.SE-dual} have the form
$\ym (\{a,b\}\cup\Zz )-\ym (\Zz )$.

The necessity of \eqref{eq.SEF} can be shown by an inductive construction. Take
$\Zz =\emptyset$ in \eqref{eq.SE-dual} and get $\obj (b\sepi\{ a\})=\obj (a\sepi \{ b\})$. One
can put $\ym (\{a,b\}) :=\obj (b\sepi\{ a\})$ for any pair of distinct $a,b\in N$.
Thus, owing to the above conventions, \eqref{eq.SEF} holds in case $|B|\leq 1$.
To confirm \eqref{eq.SEF} for $B$ with $|B|=r\geq 2$ accept the inductive hypothesis that it
holds for $B^{\prime}$ with $|B^{\prime}|\leq r-1$. The task is to define $m(D)$ for $D\subseteq N$
with $|D|=r+1$ so that \eqref{eq.SEF} holds for $B$ with $|B|\leq r$.
Having fixed such a set $D$, for any pair
of distinct elements $a,b\in D$ put $\Zz =D\setminus\{ a,b\}$
and observe from \eqref{eq.SE-dual} by means of the inductive premise:
$$
\obj (b\sepi\{ a\}\cup\Zz ) + \ym (\{a\}\cup\Zz ) - \ym (\Zz )=
\obj (a\sepi \{ b\}\cup\Zz ) +\ym (\{b\}\cup\Zz ) -\ym (\Zz )\,.
$$
The cancellation of $\ym (\Zz )$ implies the function
$b\mapsto \obj (b\sepi D\setminus\{ b\}) + \ym (D\setminus\{ b\})$ for $b\in D$ is constant
on $D$. Thus, one can put $\ym (D):= \obj (b\sepi D\setminus\{ b\}) + \ym (D\setminus\{ b\})$ for
any such $b\in D$, which verifies the inductive step.

The correspondence between $\obj$ and $\ym$ in \eqref{eq.SEF} is evidently a one-to-one linear mapping,
which implies the claim about the dimension.
\end{proof}

\begin{corollary}\rm\label{cor.SEF}
Let $\obj\in{\dv R}^{\Fai}$ be an SE objective and let $\ym\in{\dv R}^{\Cai}$ satisfy \eqref{eq.SEF}.
Then for any $T\in\Cai$ and arbitrary $b\in T$ with $R:=T\setminus\{ b\}$ one has
\begin{equation}
\sum_{\emptyset\neq K\subseteq R}\ (-1)^{|R\setminus K|}\cdot \obj (b\sepi K) =
\sum_{L\in\Cai :\, L\subseteq T}\ (-1)^{|T\setminus L|}\cdot \ym (L)\,.
\label{eq.SE-obj}
\end{equation}
In particular, the LHS of \eqref{eq.SE-obj} does not depend on the choice
of $b\in T$.
\end{corollary}

\begin{proof}
Having in mind the extension conventions from \S\,\ref{sec.formal-def} write using \eqref{eq.SEF}:
\begin{eqnarray*}
\lefteqn{\sum_{\emptyset\neq K\subseteq R}\ (-1)^{|R\setminus K|}\cdot \obj (b\sepi K) =
(-1)^{|R|}\cdot \sum_{K\subseteq R}\ (-1)^{|K|}\cdot \obj (b\sepi K)}~~~~~~~~~~~~~~~~~~~\\
~&\stackrel{\eqref{eq.SEF}}{=}&
(-1)^{|R|}\cdot \sum_{K\subseteq R}\ (-1)^{|K|}\cdot\left\{ \,\ym (\{ b\}\cup K) - \ym (K)\,\right\}\\
&=& (-1)^{|R|}\cdot (-1)\cdot \sum_{K\subseteq R} (-1)^{|K|+1}\cdot \ym (\{ b\}\cup K)\\
&& ~~~~~~~~+ (-1)^{|R|}\cdot (-1)\cdot\sum_{K\subseteq R} (-1)^{|K|}\cdot \ym (K)\\
&=&  (-1)^{|T|}\cdot \sum_{L\subseteq T} (-1)^{|L|}\cdot \ym (L) =
\sum_{L\in\Cai :\, L\subseteq T} (-1)^{|T\setminus L|}\cdot \ym (L)\,,
\end{eqnarray*}
which concludes the proof.
\end{proof}

Another relevant observation is the following.

\begin{lemma}\rm \label{lem.full}
Any face of\/ $\FVP$ containing the whole Markov equivalence class of full graphs is given by an SE objective.
\end{lemma}

\begin{proof}
Assume $\langle\obj ,\fam\rangle\leq u$ is an arbitrary defining inequality for such a face $F$ of\/ $\FVP$,
with $\obj\in {\dv R}^{\Fai}$, $u\in {\dv R}$. By Lemma \ref{lem.SEF}(a), it is enough to show $\obj$
satisfies \eqref{eq.SE-dual}. Note that, for any $\Zz\subseteq N$ and distinct $a,b\in N\setminus\Zz$,
full graphs $G$ and $H$ over $N$ exist with
$\fam_{G}-\fam_{H}=\Id_{b\separ \{ a\}\cup\Zz} + \Id_{a\separ\Zz}-\Id_{a\separ \{ b\}\cup\Zz} - \Id_{b\separ\Zz}$.
Hence, $\langle\obj ,\fam_{G}\rangle=u=\langle\obj ,\fam_{H}\rangle$ implies that \eqref{eq.SE-dual}
is true for that particular $a,b$ and $\Zz$.
\end{proof}

It follows from Lemma \ref{lem.full} that every face of $\FVP$ which contains the class of full graphs is an SE face.
In particular, no counter-example to Conjecture \ref{conj.SE=WEface}
is among the faces
containing a full graph. Indeed, since they must be closed under Markov equivalence, they necessarily
contain the whole set of full graphs.

\section{Correspondence to supermodular functions}\label{sec.weak-proof}
In this section we characterize those facets of\/ $\FVP$ which contain the set of full graphs.
We show they coincide with SE facets and establish their relation to extreme supermodular functions.
\smallskip

The previous results allow us to confirm Conjecture \ref{conj.SE=WEface} for facets.

\begin{theorem}\label{thm.WE=SEfacet}
The following conditions are equivalent for a facet $\setdags\subseteq\DAG$:
\begin{itemize}
\item[(a)] $\setdags$ is closed under Markov equivalence,
\item[(b)] $\setdags$ contains the whole equivalence class of full graphs,
\item[(c)] $\setdags$ is SE.
\end{itemize}
\end{theorem}

\begin{proof}
To show (a)$\Rightarrow$(b) note, by Lemma \ref{lem.increase-facet}(iii), that
$\setdags$ cannot contain the empty graph, since otherwise it is not closed under Markov equivalence.
Clearly, $\setdags$ must be non-empty, because otherwise it is not a facet of\/ $\FVP$.
Thus, $G\in\setdags$ exists and one can construct a full graph $H\in\DAG$ such that $G$
is a subgraph of $H$. By Corollary \ref{cor.increase-facet}, $H\in\setdags$.
Since $\setdags$ is closed under Markov equivalence, all full graphs belong
to $\setdags$. The implication (b)$\Rightarrow$(c) follows from Lemma~\ref{lem.full}.
The implication (c)$\Rightarrow$(a) was mentioned as Lemma \ref{lem.SE>WE} in \S\,\ref{sec.SE-defin}.
\end{proof}

The next step is to recall the definition of a supermodular set function.

\begin{definition}[standardized supermodular function]\\
Any vector $\ym\in {\dv R}^{\calP (N)}$ can be viewed as
a real set function $\ym : \calP (N)\to {\dv R}$.
Such a set function will be called {\em standardized\/}
if $\ym (S)=0$ for $S\subseteq N$, $|S|\leq 1$, and
{\em supermodular\/} if
\begin{equation}
\forall\, U,V\subseteq N\quad
\ym (U)+\ym (V)\leq \ym (U\cup V)+ \ym (U\cap V)\,.
\label{eq.supermodul}
\end{equation}
The following (non-negative) characteristics are ascribed to any supermodular function $\ym$: for any $a,b\in N$, $a\neq b$ and $\Zz\subseteq N\setminus\{ a,b\}$, we will denote
$$
\Delta\ym (a,b\sepi\Zz ) \,:=\,  \ym (\{ a,b\}\cup\Zz )+ \ym (\Zz ) -\ym (\{ a\}\cup\Zz )-\ym (\{ b\}\cup\Zz )\,.
$$
\end{definition}
\smallskip

It is easy to see that a set function $\ym$ is supermodular iff $\Delta\ym (a,b\sepi\Zz )\geq 0$
for any respective triplet $(a,b\sepi\Zz )$; see, for example, \cite[Theorem 24(iv)]{SK14}.
The point is that standardized supermodular functions correspond to valid inequalities
for the family-variable polytope that are tight at all full graphs.

\begin{lemma}\rm \label{lem.face-supermod}
An inequality $\langle\obj ,\fam\rangle\leq u$, where $\obj\in {\dv R}^{\Fai}$ and $u\in {\dv R}$, is
valid for all $\fam\in\FVP$ and tight at any full graph over $N$
iff it corresponds to a standardized supermodular function $\ym$ in the sense:
\begin{itemize}
\item $\obj$ is given by \eqref{eq.SEF}:
$\obj (a\sepi B) = \ym (\{ a\}\cup B)-\ym (B)$ for $a\in N$, $B\subseteq N\setminus\{ a\}$,
\item $u$ is the shared value $\langle\obj ,\fam_{H}\rangle$ for full graphs $H$ over $N$.
\end{itemize}
Moreover, the correspondence is one-to-one and preserves a conic combination.
\end{lemma}

\begin{proof}
Given such an inequality, Lemma \ref{lem.full} implies that $\obj$ is an SE objective
and Lemma~\ref{lem.SEF}(b) says it has the form \eqref{eq.SEF}. To show that $\ym$ is necessarily supermodular
observe $\Delta\ym (a,b\sepi\Zz )\geq 0$ for any $(a,b\sepi\Zz )$.
To this end, note that, given a triplet $(a,b\sepi\Zz )$, a full graph $H$ over $N$ and
$G\in\DAG$ exist such that $\fam_{H}-\fam_{G}= \Id_{b\separ \{a\}\cup\Zz} -\Id_{b\separ\Zz}$.
Indeed, consider a total order of elements in $N$ in which $\Zz$ precedes $a$ after which $b$ and $N\setminus (\{a,b\}\cup\Zz )$ follow and take $H$ as the full graph consonant with this order and $G$ is the graph obtained
from $H$ by the removal of the arrow $a\to b$. Hence,
$\langle\obj,\fam_{G}\rangle\leq u=\langle\obj,\fam_{H}\rangle$ implies
$0\leq \langle\obj,\fam_{H}-\fam_{G}\rangle = \obj (b\sepi \{a\}\cup\Zz ) -\obj (b\sepi\Zz )\stackrel{\eqref{eq.SEF}}{=}\Delta\ym (a,b\sepi\Zz )$.

Conversely, given a supermodular $\ym$, Lemma~\ref{lem.SEF}(b) says the objective $\obj$ given by \eqref{eq.SEF}
is SE and the full graphs $H$ over $N$ share the value $\langle\obj ,\fam_{H}\rangle$. Thus, it is enough to show
that, for any $G\in\DAG$, a full graph $H$ exists with $\langle\obj,\fam_{G}\rangle\leq\langle\obj,\fam_{H}\rangle$.
Indeed, consider a total order consonant with $G$, denote by $\pr (a)$ the set of (strict) predecessors of
$a\in N$ in that order and by $H$ the full graph consonant with the order. Then write
\begin{eqnarray*}
\lefteqn{\langle\obj,\fam_{H}-\fam_{G}\rangle = \sum_{a\in N}
\left\{\, \obj (a\sepi \pr (a)) -\obj (a\sepi \pa_{G}(a)) \,\right\} =}\\
&=& \sum_{a\in N}
\{\, \underbrace{\ym (\{ a\}\cup\pr (a)) -\ym (\{ a\}\cup\pa_{G} (a)) -\ym (\pr (a))
+\ym (\pa_{G}(a))}_{\geq 0} \,\}\geq 0\,.
\end{eqnarray*}
Since \eqref{eq.SEF} defines an invertible linear transformation,
the last claim is easy.
\end{proof}

By the extension convention from \S\,\ref{sec.formal-def}, any vector in $\ym\in {\dv R}^{\Cai}$
could be identified with a standardized set function. With a small abuse of terminology,
we say that $\ym\in {\dv R}^{\Cai}$ is supermodular if its zero extension
$\ym :\calP (N)\to {\dv R}$ is a supermodular set function.
By its definition, the set of supermodular vectors in ${\dv R}^{\Cai}$ is
a polyhedral cone. Since it is pointed, it has finitely many extreme rays.
This motivates the next definition.

\begin{definition}[extreme supermodular function]\\
A standardized supermodular set function $\ym :\calP (N)\to {\dv R}$
is called {\em extreme\/} if it generates an extreme ray of the standardized supermodular cone.
\end{definition}

The following fact follows from a specific characterization
of extremality of supermodular functions.

\begin{lemma}\rm \label{lem.ext-sup-imcompar}
Let $\ym_{1},\ym_{2}\in {\dv R}^{\calP (N)}$ generate distinct extreme rays of the standardized supermodular cone.
Then the faces of\/ $\FVP$ determined by the corresponding inequalities, as described in Lemma \ref{lem.face-supermod}, are inclusion-incomparable.
\end{lemma}

\begin{proof}
The argument is based on the result saying that a supermodular set function $\ym$ is extreme
iff the structural independence model produced by $\ym$ is sub-maximal; see \cite[Lemma 5.6]{Stu05}
or \cite[Corollary 30]{SK14}. More specifically, it says $\ym$ is extreme iff
any supermodular function $\ym^{\prime}$ with
$$
\forall\, (a,b\sepi\Zz )\quad \Delta\ym (a,b\sepi\Zz )=0 ~~\Rightarrow~~ \Delta\ym^{\prime} (a,b\sepi\Zz )=0
$$
either satisfies, for any triplet $(a,b\sepi\Zz )$, $\Delta\ym^{\prime} (a,b\sepi\Zz )=0 ~\Leftrightarrow~ \Delta\ym (a,b\sepi\Zz )=0$
or even $\Delta\ym^{\prime} (a,b\sepi\Zz )=0$, which, for a standardized $\ym^{\prime}$, means that $\ym^{\prime}$
must be a non-negative multiple of $\ym$. Since $m_{1}, m_{2}$ generate distinct rays, a triplet
$(a,b\sepi\Zz )$ must exist such that $\Delta\ym_{1}(a,b\sepi\Zz )>0$ and $\Delta\ym_{2}(a,b\sepi\Zz )=0$.
As in the proof of Lemma~\ref{lem.face-supermod}, construct a full graph $H$ over $N$ and
$G\in\DAG$ with $\fam_{H}-\fam_{G}= \Id_{b\separ \{a\}\cup\Zz} -\Id_{b\separ\Zz}$.
Then $\Delta\ym_{1}(a,b\sepi\Zz )>0$ implies that the inequality $\langle\obj_{1} ,\fam\rangle\leq u_{1}$
given by $\ym_{1}$ through \eqref{eq.SEF} is not tight for $\fam_{G}$ because
$$
u_{1}-\langle\obj_{1} ,\fam_{G}\rangle =\langle\obj_{1},\fam_{H}-\fam_{G}\rangle =
\obj_{1} (b\sepi \{a\}\cup\Zz ) -\obj_{1} (b\sepi\Zz )\stackrel{\eqref{eq.SEF}}{=}\Delta\ym_{1} (a,b\sepi\Zz )>0\,,
$$
while $\Delta\ym_{2}(a,b|\sepi\Zz )=0$ implies that $\langle\obj_{2} ,\fam\rangle\leq u_{2}$
is tight for $\fam_{G}$. Hence, the face of\/ $\FVP$ determined by $\ym_{2}$ is not contained is
the one determined by $\ym_{1}$. The role of generators $\ym_{1}$ and $\ym_{2}$ is clearly exchangeable.
\end{proof}

Now, thanks to Theorem \ref{thm.WE=SEfacet}(b), we are ready to characterize SE facets.

\begin{theorem}\label{thm.SEfacet=extsupermo}
An inequality $\langle\obj ,\fam\rangle\leq u$ for $\fam\in\FVP$,
where $\obj\in {\dv R}^{\Fai}$ and $u\in {\dv R}$, is
facet-defining for\/ $\FVP$ and tight at all full graphs over $N$ iff there exists
an extreme standardized supermodular set function $\ym$ such that
$\obj$ is determined by \eqref{eq.SEF} and $u$ is the shared value
of\/ $\langle\obj ,\fam_{H}\rangle$ for full graphs $H$ over $N$.
\end{theorem}

\begin{proof}
First, using Lemma~\ref{lem.facet=maxim}, we show that any extreme standardized supermodular function $\ym_{i}$ gives a facet of\/ $\FVP$.
Thus, assume $F^{\prime}$ is a face containing the face $F_{i}$ determined by $\ym_{i}$.
Lemma \ref{lem.face-supermod} applied to $\ym_{i}$ says that the face $F_{i}$ contains the class
of full graphs, and so $F^{\prime}$ does. Again by Lemma \ref{lem.face-supermod} applied to
the inequality defining $F^{\prime}$, the face $F^{\prime}$ is given by a supermodular
function $\ym^{\prime}$, which must be a conic combination of finitely many generators of
(all) the extreme rays: $\ym^{\prime}=\sum_{j} \alpha_{j}\cdot \ym_{j}$, $\alpha_{j}\geq 0$.
The assumption $F_{i}\subseteq F^{\prime}$ implies that, for any $k\neq i$, the coefficient $\alpha_{k}$ must vanish.
Indeed, by the last claim in Lemma \ref{lem.face-supermod}, $\alpha_{k}>0$ forces $F^{\prime}\subseteq F_{k}$:
the inequality defining $F^{\prime}$ is a conic combination of the inequality corresponding to
$\ym_{k}$ (= defining $F_{k}$) and the inequality corresponding to
$\sum_{j\neq k} \alpha_{j}\cdot \ym_{j}$ and using the arguments in the proof
of Corollary \ref{cor.facet=maxim} observe $F^{\prime}\subseteq F_{k}$. However, for distinct $i$ and $k$,
the respective faces $F_{i}$ and $F_{k}$ are inclusion-incomparable, by Lemma~\ref{lem.ext-sup-imcompar}.
Thus, either $\ym^{\prime}=0$, in which case $F^{\prime}=\FVP$, or $\ym^{\prime}$ a positive multiple of $\ym_{i}$,
in which case $F^{\prime}=F_{i}$.

Second, we show that any facet $F$ of\/ $\FVP$ involving all full graphs
is given an extreme standardized supermodular function.
Apply Lemma \ref{lem.face-supermod} to $F$ and write the respective standardized supermodular function
$\ym$ as a conic combination $\ym =\sum_{j} \alpha_{j}\cdot\ym_{j}$, $\alpha_{j}\geq 0$ of extreme ones.
Let us assume for a contradiction that $\alpha_{i}\neq 0\neq\alpha_{k}$ for distinct $i$ and $k$.
By Lemma \ref{lem.ext-sup-imcompar}, the faces corresponding to $\ym_{i}$ and $\ym_{k}$ are incomparable.
In particular, provided $\langle\obj_{j},\fam\rangle\leq u_{j}$ denotes the inequality for $\fam\in\FVP$
corresponding $\ym_{j}$, we know that $w_{1}\in\FVP$ exists satisfying
$\langle\obj_{i},w_{1}\rangle =u_{i}$ and $\langle\obj_{k},w_{1}\rangle <u_{k}$,
and $w_{2}\in\FVP$ exists satisfying $\langle\obj_{i},w_{2}\rangle <u_{i}$.
The inequality corresponding to $\ym$ is the sum of
$\sum_{j\neq i} \alpha_{j}\cdot\langle\obj_{j},\fam\rangle\leq \sum_{j\neq i} \alpha_{j}\cdot u_{j}$
and of the $\alpha_{i}$-multiple of the inequality $\langle\obj_{i},\fam\rangle\leq u_{i}$.
The assumption \eqref{eq.non-facet} of Corollary~\ref{cor.facet=maxim} is fulfilled
for the vectors $w_{1}$ and $w_{2}$ above, which gives a contradictory conclusion that $F$ is not a facet.
Thus, at most one of the coefficients $\alpha_{j}$ is non-zero. Since $\ym$ must be non-zero,
it is a positive multiple of some $\ym_{j}$.
\end{proof}

Thus, Theorem~\ref{thm.SEfacet=extsupermo} transforms the problem of testing
certain facets of\/ $\FVP$ into the task to verify whether
the respective supermodular function is extreme. Note that
a simple linear criterion for testing extremality of a standardized supermodular
function $\ym$ has recently been proposed in \cite{SK14}. The criterion consists
in solving a linear equation system determined by the combinatorial structure
of the so-called {\em core polytope\/} ascribed to $\ym$:
$$
{\cal C}(\ym ) := \{\, [v_{a}]_{a\in N}\in {\dv R}^{N}\, : ~ \sum_{a\in N} v_{a}=\ym (N) \,~\&~\,
\forall\, S\subseteq N ~\sum_{a\in S} v_{a}\geq\ym (S)\,\}\,.
$$
We hope that the criterion from \cite{SK14} will appear to be useful in our context.

\section{Generalized cluster inequalities and uniform matroids}\label{sec.cluster-ineq}

An important class of inequalities for the family-variable
polytope is discussed in this section. We apply
Theorem~\ref{thm.SEfacet=extsupermo} from the previous section 
to show they define SE facets and reveal their hidden connection to uniform matroids.
\smallskip

Jaakkola, Sontag, Globerson and Meila introduced in \cite{JSGM10} an interesting class 
of cluster-based inequalities for $\FVP$, whose purpose was to express the acyclicity restrictions.
To shorten the terminology we call them the {\em cluster inequalities}.
Specifically, if the family-vector $\fam_{G}$ encoding $G\in\DAG$ is extended by additional components
for the empty parent sets $\fam_{G}(a\sepi\emptyset)$, $a\in N$, then
the inequality ascribed to a cluster $C\subseteq N$, $|C|\geq 2$, has the form
$$
1\leq\, \sum_{a\in C}\ \sum_{B\subseteq N\,:\, B\cap C=\emptyset}\ \fam_{G} (a\sepi B)\,.
$$
The interpretation is clear: since the induced subgraph $G_{C}$ is acyclic,
there is at least one node $a$ in $C$ which has no parent in $C$.
An important fact is that the only integral vectors in the polyhedron specified by the cluster inequalities,
and, for any $a\in N$, by the {\em convexity constraints\/} $\fam_{G}(a\sepi B)\geq 0$, $B\subseteq N\setminus\{ a\}$
and $\sum_{B\subseteq N\setminus\{ a\}} \fam_{G} (a\sepi B)=1$,
are the DAG-codes \cite[Lemma\,2]{SH13}.
\smallskip

The cluster inequalities have appeared to have a crucial role in the integer linear programming (ILP) approach
learning BN structure. This was confirmed computationally in \cite{Cus11} 
by the first author of this paper, who also introduced {\em generalized cluster inequalities}.
Specifically, to every cluster $C\subseteq N$, $|C|\geq 2$, and
$k=1,\ldots ,|C|-1$ one can ascribe the inequality
$$
k\leq\, \sum_{a\in C}\ \sum_{B\subseteq N\setminus\{ a\}\,:\, |B\cap C|<k}\ \fam_{G} (a\sepi B)\,.
$$
Its interpretation is analogous: since the induced subgraph $G_{C}$ is acyclic,
the first $k$ nodes in a total order of nodes in $C$ consonant with $G_{C}$ have at most $k-1$
parents in $C$. Note that for $k=|C|$ and $k=0$ the inequalities are tight at any $G\in\DAG$ and are, therefore,
omitted. In particular, we only consider the (generalized) cluster inequalities for $k=1,\ldots ,|C|-1$;
this also enforces $|C|\geq 2$.
To transform them into standardized inequality constraints on a vector $\fam$ in $\FVP\subseteq {\dv R}^{\Fai}$
we use the above convexity equality constraints and get for any $C\subseteq N$, $|C|\geq 2$, and $k=1,\ldots ,|C|-1$,
\begin{equation}
\sum_{a\in C}\ \sum_{B\subseteq N\setminus\{ a\}\,:\, |B\cap C|\geq k}\ \fam (a\sepi B)\,\leq\, |C|-k\,.
\label{eq.gen-cluster}
\end{equation}
The point is that this $k$-cluster inequality \eqref{eq.gen-cluster} corresponds to
an extreme standardized supermodular set function in sense of Theorem \ref{thm.SEfacet=extsupermo}.

\begin{lemma}\label{lem.cluster-supermod}\rm
For any $C\subseteq N$, $|C|\geq 2$, and $k=1,\ldots ,|C|-1$, the formula
\begin{equation}
\ym_{C,k}(S) = \max\, \{\, 0, |S\cap C|-k\,\} \quad \mbox{for any}~~ S\subseteq N,
\label{eq.cluster-supermod}
\end{equation}
gives an extreme standardized supermodular function which determines through the formula \eqref{eq.SEF} the objective
coefficients in \eqref{eq.gen-cluster}.
\end{lemma}

\begin{proof}
Easily, the objective coefficient for $(a\sepi B)\in\Fai$ is
$$
\obj_{C,k}(a\sepi B) \stackrel{\eqref{eq.SEF}}{=}
\ym_{C,k}(\{ a\}\cup B)-\ym_{C,k}(B) =
\left\{
\begin{array}{ll}
1 & ~~~~\mbox{if $a\in C$ and $|B\cap C|\geq k$,}\\
0 & ~~~~\mbox{otherwise},
\end{array}
\right.
$$
and the value of $\langle\obj_{C,k},\fam_{H}\rangle$ for any full graph $H$ over $N$ is
$|C|-k$. Hence, \eqref{eq.cluster-supermod} determines through \eqref{eq.SEF} the inequality
\eqref{eq.gen-cluster}.

It remains to show that $\ym_{C,k}$ generates an extreme ray of the cone $\KON$ of standardized supermodular functions. Recall $\ym$ is supermodular iff,
for any triplet $A,B,\Zz\subseteq N$ of pairwise disjoint sets, one has
$$
\Delta\ym (A,B\sepi\Zz ) \,:=\,  \ym (A\cup B\cup\Zz )+ \ym (\Zz ) -\ym (A\cup\Zz )-\ym (B\cup\Zz )\geq 0\,,
$$
which is a re-formulation of \eqref{eq.supermodul}, but it enough to verify
$\Delta\ym (a,b\sepi\Zz )\geq 0$ for any $a,b\in N$, $a\neq b$ and $\Zz\subseteq N\setminus\{ a,b\}$.
It is easy to observe $\ym (S)\geq 0$ for any $\ym\in\KON$ and $S\subseteq N$.
Since $\ym_{C,k} (S)=\ym_{C,k} (S\cap C)$ for any $S\subseteq N$, one has
$$
\Delta\ym_{C,k} (A,B\sepi\Zz ) =
\Delta\ym_{C,k} (A\cap C,B\cap C\sepi\Zz\cap C )\quad
\mbox{for disjoint $A,B,\Zz\subseteq N$}.
$$
To show $\ym_{C,k}\in\KON$ observe that, for any triplet $(a,b\sepi\Zz )$ with $\{a,b\}\cup\Zz\subseteq C$,
$$
\begin{array}{lcl}
\Delta\ym_{C,k} (a,b\sepi\Zz )=1 &~~& \mbox{if $|\{a,b\}\cup\Zz|=k+1$, and}\\
\Delta\ym_{C,k} (a,b\sepi\Zz )=0 && \mbox{otherwise.}
\end{array}
$$
We have to verify that,
if $\ym_{C,k} =\alpha\cdot\ym_{1}+ (1-\alpha)\cdot\ym_{2}$, $\alpha\in (0,1)$ is a
non-trivial convex combination of $\ym_{1},\ym_{2}\in \KON$ then $\ym_{1}$ and $\ym_{2}$ are non-negative
multiples of $\ym_{C,k}$. To show $\ym =\gamma\cdot\ym_{C,k}$ for
some $\gamma\geq 0$ it is enough to verify:
\begin{description}
\item[(i)] ~\,$\ym (S)=0$ for $S\subseteq N$ with $|S\cap C|\leq k$,
\item[(ii)] ~$\ym (S)=\ym (S\cap C)$ for any $S\subseteq N$,
\item[(iii)] $\ym (S)=\ym (T)$ for $S,T\subseteq C$, $|S|=|T|=k+1$,
\item[(iv)] if $\gamma$ is the shared value from (iii) then $\ym (S)=\gamma+\ym (R)$
for any $R,S\subseteq C$, such that $R\subseteq S$ and $k\leq |R|=|S|-1$.
\end{description}
To verify (i) for $\ym_{1},\ym_{2}$ with some such $S\subseteq N$ write
$$
0=\ym_{C,k} (S)=\alpha\cdot\ym_{1}(S)+(1-\alpha)\cdot\ym_{2}(S)\,.
$$
The RHS here is a convex combination of non-negative terms; therefore, they both vanish, which means $0=\ym_{1}(S)=\ym_{2}(S)$.
To verify (ii) for $\ym_{1},\ym_{2}$ with some $S\subseteq N$
consider $(A,B\sepi\Zz ) = (S\cap C,S\setminus C\sepi\emptyset )$ and observe
$$
0=\Delta\ym_{C,k} (A,B\sepi\Zz ) =
\alpha\cdot \Delta\ym_{1} (A,B\sepi\Zz ) + (1-\alpha)\cdot \Delta\ym_{2} (A,B\sepi\Zz )\,.
$$
Hence, for $i=1,2$, $\Delta\ym_{i} (A,B\sepi\Zz )=0$, implying together with (i) for $\ym_{i}$ that
$\ym_{i}(S)=\ym_{i}(S\cap C)$. To verify (iii) it is enough to observe $\ym_{i}(S)=\ym_{i}(T)$ in case
$|S|=|T|=k+1$ with $S\setminus T=\{ s\}$ and $T\setminus S=\{ t\}$.
Choose $r\in S\cap T$, put
$R=(S\cap T)\setminus\{ r\}$ and consider the triplets  $(r,t\sepi R\cup\{ s\})$ and $(r,s\sepi R\cup\{ t\})$.
Since both $0=\Delta\ym_{C,k}(r,t\sepi R\cup\{ s\})$ and $0=\Delta\ym_{C,k}(r,s\sepi R\cup\{ t\})$, one has
$0=\Delta\ym_{i}(r,t\sepi R\cup\{ s\})=\Delta\ym_{i}(r,s\sepi R\cup\{ t\})$, for $i=1,2$.
Hence, by (i) for $\ym_{i}$,
$$
0=\Delta\ym_{i}(r,t\sepi R\cup\{ s\})-\Delta\ym_{i}(r,s\sepi R\cup\{ t\})= \ym_{i}(T)-\ym_{i}(S)\,.
$$
The condition (iv) can be verified by induction on $|S|$: (i) and (iii) for $\ym_{i}$
say (iv) holds for $|S|=k+1$. If $|S|>k+1$ and $S\setminus R=\{ s\}$ choose $t\in R$ and put $T=S\setminus\{ t\}$.
Because $0=\Delta\ym_{C,k}(s,t\sepi R\cap T)$ one gets
$0=\Delta\ym_{i}(s,t\sepi R\cap T)$, that is, $\ym_{i}(S)-\ym_{i}(R)=\ym_{i}(T)-\ym_{i}(R\cap T)=\gamma$
by the inductive assumption.
\end{proof}

\begin{corollary}\label{cor.cluster}\rm
Any generalized cluster inequality \eqref{eq.gen-cluster} defines an SE facet of\/ $\FVP$.
\end{corollary}

\begin{proof}
Combine Lemma \ref{lem.cluster-supermod} with Theorems \ref{thm.SEfacet=extsupermo} and \ref{thm.WE=SEfacet}.
\end{proof}
\smallskip

The rest of this section is an observation which makes sense for a reader
familiar with elementary notions in the matroid theory. Thus, we assume
the reader knows basic equivalent definitions of a matroid in terms
of independent sets, bases and the rank function, as given, for example, in \cite[Chapter~1]{Oxl92}.
\smallskip

The link between generalized cluster inequalities and certain matroids
is based on a duality relationship of supermodular functions and their mirror images, submodular
functions. In fact, there is a one-to-one linear mapping from the cone $\KON$ of standardized
supermodular functions onto the cone of submodular functions $r: \calP (N)\to {\dv R}$
satisfying $r(\emptyset )=0$ and $r(N)=r(N\setminus\{ a\})$ for any $a\in N$.
The point is that the rank functions of non-degenerate matroids fall within this submodular cone.
Specifically, one can consider the {\em duality transformation} which
ascribes to any $\ym\in\KON$ the set function $r$ given by
$$
r(T) = \ym (N) - \ym (N\setminus T)\qquad
\mbox{for any $T\subseteq N$.}
$$
This self-inverse transformation maps the supermodular function $\ym_{C,k}$ for $C\subseteq N$, $|C|\geq 2$,
and $k=1,\ldots ,|C|-1$, onto the submodular function
\begin{equation}
r_{C,k}(T) = \min\, \{\, |T\cap C|, |C|-k\,\} \quad \mbox{for any}~~ T\subseteq N,
\label{eq.rank-uniform}
\end{equation}
which is the rank function of a matroid on $N$. However, it can be viewed as a
kind of trivial ``loop-adding" extension of a matroid which has $C$ as its ground set.
Indeed, the function \eqref{eq.rank-uniform}
can be identified with its restriction to $\calP (C)$, which is the rank function of the
uniform matroid of rank $|C|-k$ on $C$; see \cite[Example~1.2.7]{Oxl92}. The bases of this matroid
are just the subsets of $C$ of the cardinality $|C|-k$. Two remaining uniform matroids on $C$,
namely those of the ranks $0$ and $|C|$, differ in the property they are not connected:
that means a set $\emptyset\subset S\subset C$ exists with $r(C)=r(S)+r(C\setminus S)$,
where $r$ is their rank function; see \cite[\S\,4.2]{Oxl92} for this concept.
Therefore, one can summarize our observation by saying that the generalized cluster inequalities
for $C\subseteq N$ are in a one-to-one correspondence with {\em connected uniform matroids\/} on $C$.

\begin{remark}\rm
Note that the duality transformation is not the only one-to-one linear mapping between
the considered supermodular and submodular cones; see \cite[\S\,7.2]{SK14} for the details.
However, this fact is not important in our context since the use of the other transformation
leads to the same conclusion, the difference is that the uniform matroid on $C$ of the rank $k$
is ascribed to $\ym_{C,k}$ instead. On the other hand, the duality transformation has the property that
the vertices of the core polytope ascribed to $\ym_{C,k}$, as defined in the end of \S\,\ref{sec.weak-proof},
are just the incidence vectors for bases of the uniform matroid of rank $|C|-k$.
\end{remark}

\section{On the faces of the characteristic-imset polytope}\label{sec.fac-cores}

In this section, we introduce a one-to-one correspondence between faces of the characteristic-imset polytope
$\CIP$ and SE faces of the family-variable \mbox{polytope $\FVP$}. This allows us to characterize those
faces of $\CIP$ that correspond to SE facets.
\smallskip

Let $\langle\zob ,\c\rangle_{\Cai}\leq u$, where $\zob\in {\dv R}^{\Cai}$ and $u\in {\dv R}$, be
a valid inequality for $\c$ in the characteristic-imset polytope $\CIP$. It defines a face of $\CIP$:
$$
\barF =\{\, \c\in\CIP \, :\ \langle\zob ,\c\rangle_{\Cai} = u\,\}\,.
$$
By substituting \eqref{eq.eta-to-char} into the inequality $\langle\zob ,\c_{\fam}\rangle_{\Cai}\leq u$
and re-arranging terms after the components of $\fam$ one gets
an inequality for $\fam\in {\dv R}^{\Fai}$ valid for any $\fam_{G}$, $G\in\DAG$.
Indeed, this is because the image of $\fam_{G}$ by \eqref{eq.eta-to-char} is just $\c_{G}$.
Moreover, the objective on the LHS of the obtained inequality is SE because whenever $G\sim H$, one has $\c_{G}=\c_{G}$
and, therefore, $\langle\zob ,\c_{G}\rangle_{\Cai} =\langle\zob ,\c_{H}\rangle_{\Cai}$.
Thus, any face of $\CIP$ defines an SE face of\/ $\FVP$. Nevertheless, the converse is true.

\begin{lemma}\rm\label{lem.fac-cores}
Given an SE objective $\obj\in {\dv R}^{\Fai}$, there exists unique
$\zob_{\obj}\in {\dv R}^{\Cai}$ such that the following holds:
\begin{equation}
\forall\, \fam\in {\dv R}^{\Fai} \qquad \langle\obj ,\fam\rangle_{\Fai} =\langle\zob_{\obj},\c_{\fam}\rangle_{\Cai}\,.
\label{eq.fac-cores}
\end{equation}
Specifically, one has
\begin{eqnarray}
\zob_{\obj}(T) &:=& \sum_{\emptyset\neq K\subseteq R}\ (-1)^{|R\setminus K|}\cdot \obj (b\sepi K)
\label{eq.Se-object-trans}\\
&& ~~~~\mbox{for $T\in\Cai$, with any $b\in T$ and $R:= T\setminus\{ b\}$.} \nonumber
\end{eqnarray}
In particular, the expression in \eqref{eq.Se-object-trans} does not depend on the choice of $b\in T$.
\end{lemma}

\begin{proof}
We are going to show that $\zob_{\obj}\in {\dv R}^{\Cai}$ given by \eqref{eq.Se-object-trans} satisfies
\begin{equation}
\forall\, G\in\DAG\qquad \langle\obj ,\fam_{G}\rangle_{\Fai} =\langle\zob_{\obj},\c_{G}\rangle_{\Cai}\,.
\label{eq.fac-cor-graph}
\end{equation}
By Corollary \ref{cor.SEF}, we know that $\zob_{\obj}$ takes the form
\begin{eqnarray}
\zob_{\obj}(T) &\stackrel{\eqref{eq.SE-obj}}{=}& \sum_{L\in\Cai\, :\, L\subseteq T}\ (-1)^{|T\setminus L|}\,\cdot\, \ym (L)\quad \mbox{for $T\in\Cai$;} \label{eq.Moeb-z-y}\\
&& \qquad\quad \mbox{where $\ym\in {\dv R}^{\Cai}$ given by \eqref{eq.SEF}.}
\nonumber
\end{eqnarray}
The next step is to note that \eqref{eq.Moeb-z-y} is equivalent to the relation
\begin{equation}
\ym (S) = \sum_{T\in\Cai\, :\, T\subseteq S}\ \zob_{\obj}(T)\quad
\mbox{for any $S\in\Cai$},
\label{eq.Moeb-y-z}
\end{equation}
which can be verified by substituting \eqref{eq.Moeb-z-y} into the RHS of \eqref{eq.Moeb-y-z}.
To verify \eqref{eq.fac-cor-graph} substitute \eqref{eq.SEF} into the expression for
$\langle\obj ,\fam_{G}\rangle_{\Fai}$, then use the definitions of $\fam_{G}$ and that of the standard imset $\u_{G}$:
\begin{eqnarray*}
\lefteqn{\langle\obj ,\fam_{G}\rangle_{\Fai}
\stackrel{\eqref{eq.SEF}}{=} \sum_{(a\sepi B)\in\Fai}\,
\left\{\, \ym (\{ a\}\cup B)-\ym (B) \,\right\} \cdot \fam_{G}(a\sepi B)}\\
&=&  \sum_{\emptyset\neq S\subseteq N} \ym (S)\cdot
\left\{ \sum_{(a\sepi B)}\, \fam_{G}(a\sepi B)\cdot\delta_{S}(\{ a\}\cup B)
- \sum_{(a\sepi B)}\, \fam_{G}(a\sepi B)\cdot\delta_{S}(B) \right\}\\
&=& \sum_{\emptyset\neq S\subseteq N} \ym (S)\cdot \left\{ \sum_{a\in N}\, \delta_{S}(\{ a\}\cup \pa_{G}(a))
- \sum_{a\in N}\, \delta_{S}(\pa_{G}(a)) \right\}\\
&=& \sum_{\emptyset\neq S\subseteq N} \ym (S)\cdot \sum_{a\in N}\ \{\, \delta_{\{ a\}\cup \ppa_{G}(a)}(S) - \delta_{\ppa_{G}(a)}(S)\,\}\\
&\stackrel{\eqref{eq.def-stand}}{=}& \sum_{\emptyset\neq S\subseteq N} \ym (S)\cdot \{\delta_{N}(S)-\u_{G}(S)\}\,.
\end{eqnarray*}
Further, we substitute the relation \eqref{eq.Moeb-y-z} into the above expression and get this:
\begin{eqnarray*}
\lefteqn{\langle\obj ,\fam_{G}\rangle_{\Fai} = \sum_{\emptyset\neq S\subseteq N} \ym (S)\cdot \{\delta_{N}(S)-\u_{G}(S)\}}\\
~~&\stackrel{\eqref{eq.Moeb-y-z}}{=}&
\sum_{S\in\Cai}\,\,  \sum_{T\in\Cai\,:\, T\subseteq S} \zob_{\obj}(T)\cdot \{\delta_{N}(S)-\u_{G}(S)\} \\
&=& \sum_{T\in\Cai}\ \zob_{\obj}(T)\cdot \sum_{S:\, T\subseteq S\subseteq N} \{\delta_{N}(S)-\u_{G}(S)\} \\
&=& \sum_{T\in\Cai}\ \zob_{\obj}(T)\cdot \{ 1- \sum_{S:\, T\subseteq S\subseteq N} \u_{G}(S)\}
\stackrel{\eqref{eq.stan-to-char}}{=} \sum_{T\in\Cai}\ \zob_{\obj}(T)\cdot \c_{G}(T) = \langle\zob_{\obj},\c_{G}\rangle_{\Cai}\,.
\end{eqnarray*}
Since the codes $\fam_{G}$ for $G\in\DAG$ linearly span ${\dv R}^{\Fai}$ the relation \eqref{eq.fac-cor-graph}
implies \eqref{eq.fac-cores}. The uniqueness of the vector $\zob_{\obj}$ in the formula \eqref{eq.fac-cores}
is easy because the codes $\c_{G}$ for $G\in\DAG$ span ${\dv R}^{\Cai}$.
\end{proof}

Every face of the characteristic-imset polytope $\CIP$ can be identified with a set of acyclic directed graphs
closed under Markov equivalence:
$$
\barF\subseteq\CIP ~~\mbox{a face of\/ $\CIP$} ~~\longleftrightarrow ~~
\setdags =\{ G\in\DAG :\ \c_{G}\in\barF \}\,.
$$
Indeed, the arguments given above Definition~\ref{def.face-DAG} are also
valid for $\POL =\CIP$ and, since the vertices of $\CIP$ are just the characteristic imsets,
its faces can be viewed as sets of characteristic imsets. These, however, correspond to equivalence classes of
graphs over $N$. Thus, every face of $\CIP$ can be identified
with a set of such graphs, namely with the union of the respective equivalence classes.
These are just the graphs whose characteristic imsets belong to the face.
It is easy to see that the correspondence preserves inclusion: $\barF_{1}\subseteq\barF_{2}$ for faces
of $\CIP$ iff $\setdags_{1}\subseteq\setdags_{2}$ for the corresponding sets of graphs $\setdags_{i}\subseteq\DAG$.

\begin{corollary}\label{cor.fac-cores}\rm
There is a one-to-one correspondence between SE faces of\/ $\FVP$ and faces of $\CIP$ which preserves inclusion:
given SE faces $F_{1},F_{2}$ of\/ $\FVP$ and the corresponding faces $\barF_{1},\barF_{2}$ of $\CIP$ one has
$F_{1}\subseteq F_{2}$ if and only if $\barF_{1}\subseteq \barF_{2}$.
Specifically, the SE face of\/ $\FVP$ given by an inequality $\langle\obj ,\fam\rangle_{\Fai}\leq u$
corresponds to the face of\/ $\CIP$ given the inequality $\langle\zob_{\obj},\c\rangle_{\Cai}\leq u$.
This correspondence has the property that the sets of graphs identified with the faces coincide.
\end{corollary}

\begin{proof}
It is easy to see that $\langle\obj ,\fam\rangle_{\Fai}\leq u$ is valid for $\fam\in\FVP$ iff
$\langle\zob_{\obj},\c\rangle_{\Cai}\leq u$ is valid for $\c\in\CIP$.
Moreover, by Lemma \ref{lem.fac-cores}, the set of $G\in\DAG$ such that $\langle\obj ,\fam\rangle_{\Fai}\leq u$
is tight for $\fam_{G}$ coincides with the set of $G\in\DAG$ such that $\langle\zob_{\obj},\c\rangle_{\Cai}\leq u$
is tight for $\c_{G}$. Thus, an SE face of\/ $\FVP$ and the corresponding face of $\CIP$ have the same sets of ``belonging" graphs. This observation easily implies the claim about preserving the inclusion of faces.
\end{proof}

There are two distinguished vertices of the characteristic imset polytope  $\CIP$.
One of them is the 0-imset, the zero vector in ${\dv R}^{\Cai}$, which is the
characteristic imset of the empty graph over $N$. The other one is the 1-imset,
a vector in ${\dv R}^{\Cai}$ whose all components are ones, which is the characteristic
imset of any of the full graphs over $N$. It plays a crucial role in the
description of faces of $\CIP$ corresponding to SE facets of\/ $\FVP$.

\begin{corollary}\label{cor.SE-facet-cores}\rm
SE facets of the family-variable polytope $\FVP$ correspond to those facets of
the characteristic-imset polytope $\CIP$ that contain the 1-imset.
None of those facets of $\CIP$ include the 0-imset.
\end{corollary}

\begin{proof}
Let $\barF$ be a face of $\CIP$ corresponding to an SE facet $F$ of\/ $\FVP$.
We show, using Lemma~\ref{lem.facet=maxim}, that $\barF$ is a facet of $\CIP$.
For a face $\barF^{\prime}$ of $\CIP$ with $\barF\subset\barF^{\prime}$ the respective SE face $F^{\prime}$
of\/ $\FVP$ satisfies, by Corollary~\ref{cor.fac-cores}, $F\subset F^{\prime}$. Thus, necessarily
$F^{\prime}=\FVP$, which implies $\barF^{\prime}=\CIP$. By Theorem \ref{thm.WE=SEfacet},
$F$ contains the whole equivalence class of full graphs and,  by Corollary~\ref{cor.fac-cores}, $\barF$ must
contain the 1-imset.

Conversely, let $F$ be an SE face of\/ $\FVP$ which corresponds to a facet $\barF$ of $\CIP$
containing the 1-imset. Using Lemma~\ref{lem.facet=maxim} observe that $F$ is a facet of\/ $\FVP$.
Indeed, since $F$ contains the whole equivalence class of full graphs, the same is the case for
any face $F^{\prime}$ of\/ $\FVP$ with $F\subset F^{\prime}$. By Lemma~\ref{lem.full}, $F^{\prime}$ is SE;
hence, it has the corresponding face $\barF^{\prime}$ of $\CIP$. By Corollary~\ref{cor.fac-cores} one has $\barF\subset \barF^{\prime}$;
therefore, $\barF^{\prime}=\CIP$, which implies $F^{\prime}=\FVP$.

The last claim follows easily by contradiction: otherwise the corresponding SE facet contains the empty graph
and, by Lemma \ref{lem.increase-facet}(iii), it is determined by \eqref{eq.facet-empty}.
But none of these facets of\/ $\FVP$ is SE.
\end{proof}

By combining Corollary \ref{cor.SE-facet-cores}, Theorems \ref{thm.WE=SEfacet} and \ref{thm.SEfacet=extsupermo}
one observes that the facets of $\CIP$ containing the 1-imset correspond to extreme supermodular functions.
On the other hand, it follows from Corollaries \ref{cor.fac-cores} and \ref{cor.SE-facet-cores} that the
SE faces of\/ $\FVP$ corresponding to facets of $\CIP$ not containing the 1-imset are
{\em sub-maximal SE faces} with respect to inclusion, but {\em not SE facets}. That means, these are SE faces
$F$ of\/ $\FVP$ such that there is no other SE face $F^{\prime}$ of \/ $\FVP$ such that
$F\subset F^{\prime}$ except $F^{\prime}=\FVP$ but $F$ is not a facet of\/ $\FVP$ since
$\dim (F)<\dim (\FVP )-1$. Example~\ref{exa.three-CIP}
in \S\,\ref{sec.examples} shows what such sub-maximal SE faces look like.
\smallskip

To illustrate Corollary \ref{cor.fac-cores} we transform the generalized cluster inequalities
\eqref{eq.gen-cluster} from \S\,\ref{sec.cluster-ineq} into the characteristic-imset frame.
Specifically, having fixed a cluster $C\subseteq N$, $|C|\geq 2$ and $k\in \{1,\ldots , |C|-1\}$,
the coefficients $\zob (S)$ for $S\in\Cai$ in the transformed corresponding
$k$-cluster inequality vanish outside subsets of $C$ and only depend
on the cardinality of the set $S$:
\begin{equation}
\zob (S) =
\left\{
\begin{array}{cl}
(-1)^{|S|-k-1}\cdot {|S|-2\choose |S|-k-1} & ~\mbox{if $S\subseteq C$ and $|S|\geq k+1$,}\\
0 & ~\mbox{otherwise}.
\end{array}
\right.
 ~\mbox{for $S\in\Cai$.}
\label{eq.tau-formula}
\end{equation}
The proof is based on an auxiliary combinatorial identity \eqref{eq.iden-1} from \S\,\ref{sec.comb-iden}.

\begin{lemma}\label{lem.gen-clust-char}\rm
In the context of the characteristic-imset polytope, the $k$-cluster inequality \eqref{eq.gen-cluster}
for $C\subseteq N$, $|C|\geq 2$ and $k\in \{1,\ldots , |C|-1\}$, takes the form
\begin{equation}
\sum_{S\in\Cai}\ \zob (S)\cdot \c(S)\leq |C|-k\,,
\quad \mbox{where $\zob (S)$ are given by \eqref{eq.tau-formula}.}
\label{eq.tau-ineq}
\end{equation}
\end{lemma}

\begin{proof}
By suitable substitutions we re-write \eqref{eq.tau-ineq} into the desired form \eqref{eq.gen-cluster}:
\begin{eqnarray*}
\lefteqn{\sum_{S\in\Cai}\, \zob (S)\cdot \c_{\fam}(S) \stackrel{\eqref{eq.tau-formula},\eqref{eq.eta-to-char}}{=}
\sum_{S\subseteq C\,:\, |S|\geq k+1} \zob (S)\cdot
\sum_{a\in S} \sum_{B\,:\, S\setminus\{ a\}\subseteq B\subseteq N\setminus\{ a\}} \fam (a\sepi B)\quad ~}\\
&=& \sum_{a\in C}\, \sum_{B\subseteq N\setminus\{ a\}\,:\, |B\cap C|\geq k}
\fam (a\sepi B)\cdot
\underbrace{\sum_{S\,:\,|S|\geq k+1,\, a\in S,\, S\setminus\{ a\}\subseteq B\cap C}
\zob (S)}_{1}\,.
\end{eqnarray*}
It remains to show that, for fixed $a\in C$ and $B\subseteq N\setminus\{ a\}$ with $|B\cap C|\geq k$,
the indicated expression is indeed $1$. We put $\ell :=|B\cap C|$, $s:=\ell -k$ and write:
\begin{eqnarray*}
\lefteqn{\sum_{S\,:\, |S|\geq k+1,\, a\in S,\, S\setminus\{ a\}\subseteq B\cap C}
\zob (S)~ = \sum_{R\subseteq B\cap C\,:\, |R|\geq k} \zob (\{ a\}\cup R)}\\
&\stackrel{\eqref{eq.tau-formula}}{=}&
\sum_{R\subseteq B\cap C,\, |R|\geq k} (-1)^{|R|-k}\cdot {|R|+1-2\choose |R|-k}
= \sum_{r=k}^{\ell}\ {\ell\choose r}\cdot (-1)^{r-k}\cdot {r-1\choose r-k}\\
&=& \sum_{m=0}^{\ell -k}\ {\ell\choose k+m}\cdot (-1)^{m}\cdot {m+k-1\choose m}\\
&=& \sum_{m=0}^{s}\ (-1)^{m}\cdot {k+s\choose k+m}\cdot{m+k-1\choose m}
\stackrel{\eqref{eq.iden-1}}{=} 1\,,
\end{eqnarray*}
which concludes the proof.
\end{proof}

Thus, it follows from Lemma \ref{lem.gen-clust-char} using Corollaries \ref{cor.cluster}
and \ref{cor.SE-facet-cores} that \eqref{eq.tau-ineq} defines a facet of $\CIP$ containing the 1-imset.

\section{Simple illustrating examples}\label{sec.examples}
To illustrate the achieved results we analyze completely the situation in the case of three BN variables
and comment on the case of four BN variables.
\medskip

We have observed that the following inequalities are facet-defining for the family-variable
polytope $\FVP$ in case $|N|=n\geq 3$:
\begin{itemize}
\item the {\em non-negativity\/} constraints \eqref{eq.facet-empty} (see Lemma \ref{lem.increase-facet}(iii)),
\item the {\em modified convexity\/} constraints \eqref{eq.mod-conv} (see Lemma \ref{lem.mod-convex-facet}), and
\item the {\em generalized cluster\/} inequalities \eqref{eq.gen-cluster} (see Corollary \ref{cor.cluster}).
\end{itemize}
This is a complete list of facets of\/ $\FVP$ in the case of three BN variables. The following example illustrates
the observations from \S\,\ref{sec.cluster-ineq}; we use a shorthand $\fam (a\sepi bc)$ for $\fam (a\sepi \{ b,c\})$
below.

\begin{example}\label{exa.three-FVP}\rm
If $N=\{ a,b,c\}$ one has $|\Fai |=9$. The 9-dimensional polytope $\FVP$ has 25 vertices and 17 facets.
Five of its facets are SE and are defined by the generalized cluster inequalities.
They decompose into 3 permutation types:
\begin{itemize}
\item[$\bullet$] $\fam (a\sepi b)+\fam (a\sepi bc)+\fam (b\sepi a)+\fam (b\sepi ac)\leq 1$~ (3 inequalities of this type),\\[0.5ex]
the (generalized) cluster inequality for $C=\{ a,b\}$ (and $k=1$),\\
the extreme supermodular function is $\ym_{\{ a,b\},1}=\delta_{\{ a,b,c\}}+\delta_{\{ a,b\}}$,\smallskip
\item[$\bullet$] $\fam (a\sepi bc)+\fam (b\sepi ac)+\fam (c\sepi ab)\leq 1$~ (1 inequality of this type),\\[0.5ex]
the generalized cluster inequality for $C=\{ a,b,c\}$ and $k=2$,\\
it corresponds to the extreme supermodular function $\ym_{\{ a,b,c\},2}=\delta_{\{ a,b,c\}}$,\smallskip
\item[$\bullet$] $\fam (a\sepi b)+\fam (a\sepi c)+\fam (a\sepi bc)+\fam (b\sepi a)+\fam (b\sepi c)+\fam (b\sepi ac)\\
 ~~~~~~+\fam (c\sepi a)+\fam (c\sepi b)+\fam (c\sepi ab)\leq 2$~ (1 inequality of this type),\\[0.5ex]
the (generalized) cluster inequality for $C=\{ a,b,c\}$ (and $k=1$),\\
the supermodular function is $\ym_{\{ a,b,c\},1}=2\cdot\delta_{\{ a,b,c\}}+\delta_{\{ a,b\}}+\delta_{\{ a,c\}}+\delta_{\{ b,c\}}$.
\end{itemize}
If one adds nine non-negativity constraints
\begin{itemize}
\item[$\bullet$] $-\fam(a\sepi b)\leq 0$~ (6 inequalities of this type),\medskip
\item[$\bullet$] $-\fam(a\sepi bc)\leq 0$~ (3 inequalities of this type),
\end{itemize}
to those five generalized cluster inequalities then one obtains a polytope with 28 vertices.
Besides the 25 vertices of\/ $\FVP$ it has 3 additional integral vertices of the type
$\Id_{a\separ \{ b\}}+\Id_{a\separ \{ c\}}$. By adding the modified convexity constraints
\begin{itemize}
\item[$\bullet$]
$\fam (a\sepi b)+\fam (a\sepi c)+\fam (a\sepi bc)\leq 1$~ (3 inequalities of this type),
\end{itemize}
one completes the list of facet-defining inequalities for $\FVP$.
\end{example}

In the case of four BN variables there are other facet-defining inequalities
for $\FVP$ than those given by \eqref{eq.facet-empty}, \eqref{eq.mod-conv} and \eqref{eq.gen-cluster}.
In fact,
\begin{itemize}
\item there are other SE facets than those given by clusters in \eqref{eq.gen-cluster},
\item there are facets besides the SE facets and those given by the non-negativity
constraints \eqref{eq.facet-empty} and modified convexity constraints \eqref{eq.mod-conv}.
\end{itemize}

\begin{example}\rm\label{exa.four-FVP}
If $N=\{ a,b,c,d\}$ one has $|\Fai |=28$ and
the 28-dimensional \mbox{polytope} $\FVP$ has 543 vertices and 135 facets.
There exist 37 SE facets of $\FVP$ which decompose into 10 permutation types.
In \S\,\ref{sec.SE-four} we give the list of those types. Six of those types
are the generalized cluster inequalities \eqref{eq.gen-cluster}, but the remaining four of them are not.

The substantial difference from the case of three BN variables is that the polyhedron $\FVP^{*}$ specified
by 37 SE facet-defining inequalities, 28 non-negativity constraints and 4
modified convexity constraints differs from $\FVP$. We computed the vertices of $\FVP^{*}$
and found that, besides all the 543 DAG-codes,
it has 786 additional fractional vertices in comparison with $\FVP$, which decompose into 37 permutation types.
Here we give three examples of them:
\begin{eqnarray*}
\fam_{1}&=&\frac{1}{2}\cdot \Id_{a\,\separ \{ b\}} +\frac{1}{2}\cdot \Id_{a\,\separ
\{ d\}} +\frac{1}{2}\cdot \Id_{b\,\separ \{ a, c\}} +\frac{1}{2}\cdot \Id_{c\,\separ \{
a\}}\\
&&+\,\frac{1}{2}\cdot \Id_{c\,\separ \{ b, d\}} +\frac{1}{2}\cdot
\Id_{d\,\separ \{ a, b, c\}},\\
\fam_{2}&=&\frac{1}{3}\cdot \Id_{a\,\separ \{ c\}} +\frac{1}{3}\cdot \Id_{a\,\separ
\{ d\}} +\frac{1}{3}\cdot \Id_{a\,\separ \{ b, c, d\}} +\frac{1}{3}\cdot
\Id_{b\,\separ \{ a\}} +\frac{1}{3}\cdot \Id_{b\,\separ \{ a, c, d\}}\\
&&+\,\frac{1}{3}\cdot \Id_{c\,\separ \{ b\}} +\frac{1}{3}\cdot \Id_{c\,\separ \{ d\}}
+\frac{1}{3}\cdot \Id_{c\,\separ \{ a, b\}} +\frac{1}{3}\cdot \Id_{d\,\separ \{ a, b, c\}},\\
\fam_{3}&=& \frac{1}{6}\cdot \Id_{a\,\separ \{ b\}} +\frac{1}{3}\cdot \Id_{a\,\separ
\{ d\}} +\frac{1}{3}\cdot \Id_{b\,\separ \{ c\}} +\frac{1}{3}\cdot
\Id_{b\,\separ \{ a, c, d\}} \\
&& +\,\frac{1}{3}\cdot \Id_{c\,\separ \{ a\}} +\frac{1}{3}\cdot
\Id_{c\,\separ \{ d\}} +\frac{1}{3}\cdot \Id_{c\,\separ \{ a, b, d\}}
+\frac{1}{3}\cdot \Id_{d\,\separ \{ b, c\}}.
\end{eqnarray*}
Therefore, the family-variable polytope $\FVP$ necessarily has, besides the above mentioned facets,
additional non-SE facets. There are 66 such facet-defining inequalities which
decompose into five permutation types; see \cite{CJKB15} for details.
\end{example}
\smallskip

The next example is devoted to the characteristic-imset polytope $\CIP$
and illustrates the observations from \S\,\ref{sec.fac-cores}.
In  case $|N|=3$, every facet of $\CIP$ either contains the 1-imset
or contains the 0-imset.

\begin{example}\rm\label{exa.three-CIP}\rm
If $N=\{ a,b,c\}$ one has $|\Cai |=4$. The 4-dimensional polytope $\CIP$ has 11 vertices and 13 facets;
they were already discussed in \cite[Examples~5,8]{SH13}.
There are five facet-defining inequalities tight for the 1-imset; they correspond to SE facets of $\FVP$
mentioned in Example \ref{exa.three-FVP}. Here is their overview in both
modes; they decompose into 3 permutation types:
\begin{itemize}
\item[$\bullet$] $\c (ab)\leq 1$~ (3 inequalities of this type),\\[0.5ex]
in family variables $\fam (a\sepi b)+\fam (a\sepi bc)+\fam (b\sepi a)+\fam (b\sepi ac)\leq 1$,\smallskip
\item[$\bullet$] $\c (abc)\leq 1$~ (1 inequality of this type),\\[0.5ex]
in family variables $\fam (a\sepi bc)+\fam (b\sepi ac)+\fam (c\sepi ab)\leq 1$,\smallskip
\item[$\bullet$] $\c (ab) +\c (ac) +\c (bc)-\c (abc)\leq 2$~ (1 inequality of this type)\\[0.5ex]
in family variables\\
$\fam (a\sepi b)+\fam (a\sepi c)+\fam (a\sepi bc)+\fam (b\sepi a)+\fam (b\sepi c)+\fam (b\sepi ac)\\
~~~~~~+\fam (c\sepi a)+\fam (c\sepi b)+\fam (c\sepi ab)\leq 2$.
\end{itemize}
The remaining eight facet-defining inequalities of $\CIP$ are tight for the 0-imset and
decompose into 4 permutation types:
\begin{itemize}
\item[$\bullet$] $-\c (ab)\leq 0$~ (3 inequalities of this type),\\[0.5ex]
in family variables $-\fam (a\sepi b)-\fam (a\sepi bc)-\fam (b\sepi a)-\fam (b\sepi ac)\leq 0$,\smallskip
\item[$\bullet$] $-\c (abc)\leq 0$~ (1 inequality of this type),\\[0.5ex]
in family variables $-\fam (a\sepi bc)-\fam (b\sepi ac)-\fam (c\sepi ab)\leq 0$,\smallskip
\item[$\bullet$] $-\c (ab) -\c (ac) +\c (abc)\leq 0$~ (3 inequalities of this type),\\[0.5ex]
in family variables $-\fam (a\sepi b)-\fam (a\sepi c)-\fam (a\sepi bc)-\fam (b\sepi a)-\fam (c\sepi a)\leq 0$,\smallskip
\item[$\bullet$] $-\c (ab) -\c (ac) -\c (bc) +2\cdot\c (abc)\leq 0$~ (1 inequality of this type),\\[0.5ex]
in family variables $-\fam (a\sepi b)-\fam (a\sepi c)-\fam (b\sepi a)-\fam (b\sepi c)-\fam (c\sepi a)-\fam (c\sepi b)\leq 0$.
\end{itemize}
These eight inequalities define in family variables sub-maximal SE faces of\/ $\FVP$ that are not facets:
they are implied by the non-negativity constraints. The $\fam$-polyhedron $\FVP^{\prime}$ given by all 13
above-mentioned SE inequalities in unbounded, it has a linear subspace of the dimension 5.
This polyhedron is, in fact, the pre-image of the polytope $\CIP$ by the characteristic transformation \eqref{eq.eta-to-char}.
\end{example}

As concerns the case of four BN variables, unlike the case of three BN variables, there are facets
of the characteristic-imset polytope $\CIP$ which neither contain the 0-imset nor the 1-imset.

\begin{example}\rm\label{exa.four-CIP}
In case  $N=\{ a,b,c,d\}$ one has  $|\Cai |=11$ and
the 11-dimensional polytope $\CIP$ has 185 vertices and 154 facets.
Thus, it has 358 fewer vertices than the family-variable polytope $\FVP$, but 19 more facets than $\FVP$.
Besides those 37 facets that correspond to SE facets of $\FVP$ and contain the 1-imset
(Corollary~\ref{cor.SE-facet-cores}), there exist 117 facets of $\CIP$
that do not contain the 1-imset. They decompose into 20 permutation types, which
are listed in \S\,\ref{sec.CIP-four}.

With the exception of one permutation type all these inequalities are tight for the 0-imset.
The exception is
\begin{eqnarray}
-\c (bc) -\c(bd) -\c (cd) \hspace*{5mm}&&\nonumber \\
+\c (abc) +\c (abd) +\c (acd) +2\cdot\c (bcd) -2\cdot\c (abcd) &\leq& 1,\hspace*{8mm} \label{eq.c20}\\
\mbox{in family variables}~ +\fam (a\sepi bc)+\fam (a\sepi bd)+\fam (a\sepi cd)+\fam (a\sepi bcd)\hspace*{5mm}&&\nonumber \\
-\fam (b\sepi c)-\fam (b\sepi d)-\fam (c\sepi b)-\fam (c\sepi d)-\fam (d\sepi b)-\fam (d\sepi c)
 &\leq& 1,\hspace*{5mm}\nonumber
\end{eqnarray}
consisting of 4 inequalities. The $\fam$-version of the inequality \eqref{eq.c20}, therefore,
defines an inclusion submaximal SE face of\/ $\FVP$ which is not a facet.
Clearly, \eqref{eq.c20} follows from the modified convexity constraint
$$
\fam (a\sepi b)+\fam (a\sepi c)+\fam (a\sepi d)+\fam (a\sepi bc)+\fam (a\sepi bd)+\fam (a\sepi cd)+\fam (a\sepi bcd)\leq 1
$$
and the non-negativity constraints $-\fam (a\sepi b)\leq 0$, \ldots , $-\fam (d\sepi c)  \leq 0$.
\end{example}

\section{The sufficiency of SE faces}\label{sec.revised-conj}
As explained in \S\,\ref{sec.SE-defin}, the statistical task of learning
BN structure can be turned into an LP problem to maximize an SE objective
over the family-variable \mbox{polytope} $\FVP$.  When solving such problems by means
of the tools of (integer) linear programming an important question is what are the
inequalities specifying the feasible set. The computational complexity depends
on how many inequalities we actually/potentially use, how complex they are,
how closely we are able to approximate the true feasible set, which is the family-variable
polytope $\FVP$ in our case.
\smallskip

In general, facet-defining inequalities for a polytope $\POL$ are suitable
when one maximizes a linear objective over $\POL$ \cite{Wol98}.
Since our goal is to maximize quite special linear objectives over $\FVP$ a natural
question is whether we really need all facet-defining inequalities for $\FVP$.
Indeed, the correspondence between SE faces of\/ $\FVP$ and faces of the characteristic-imset
polytope $\CIP$ explained in \S\,\ref{sec.fac-cores} allows one to transform the LP
problem to maximize an SE objective over $\FVP$ into the task to maximize a general
linear function over $\CIP$. This indicates that those facets of $\FVP$ that are not
SE are perhaps superfluous. Note that we know from Example \ref{exa.four-FVP} that
there are many non-SE facets of\/ $\FVP$ besides those given by the non-negativity
\eqref{eq.facet-empty} and modified convexity \eqref{eq.mod-conv} constraints.

On the other hand, transforming our LP problems completely into the frame of
characteristic imsets does not appear to be advantageous from the point of view of
computational complexity as observed in conclusions of \cite{SH14}. The main theoretical
reason is that simple non-negativity and modified convexity constraints
are represented in the characteristic-imset frame
by much higher number of more complex specific inequalities \cite{SH13}.

This motivates the idea of combining both polyhedral approaches to
benefit from from their different strengths. The constraints that are tight at the empty graph are clearly
better represented in the family-variable frame while the constraints that are tight
at (all) the full graphs are more naturally expressed in the characteristic-imset frame.
Why not stay in the family-variable frame, utilize \eqref{eq.facet-empty}
and \eqref{eq.mod-conv} there and combine them with SE constraints, which encode
the constraints on the characteristic-imset polytope?
\medskip

In this section we show that this is indeed possible. However, the {\em original conjecture}
we started with, namely that one can limit oneself to the inequalities defining
SE facets and the non-negativity and modified convexity constraints is false;
a counter-example in given in \S\,\ref{sec.five-var}.
\smallskip

The basic observation is that one can limit to SE \emph{faces}.

\begin{lemma}\rm\label{lem.revis-conj}
Let $\obj$ be an SE objective. Then the LP problem
to maximize $\fam\mapsto \langle\obj ,\fam\rangle_{\Fai}$
over $\fam\in {\dv R}^{\Fai}$ from the polyhedron $\FVP^{\prime}$ specified by
the inequalities defining SE faces of\/ $\FVP$ has the same optimal value as the LP problem to maximize
that function over the family-variable polytope $\FVP$.
\end{lemma}

\begin{proof}
A basic observation is that the image of $\FVP$ by the transformation \eqref{eq.eta-to-char}
is $\CIP$, which can be viewed as the polyhedron specified through its faces.
The pre-image of $\CIP$ with \eqref{eq.eta-to-char} is, therefore, the polyhedron $\FVP^{\prime}$
of $\fam$-vectors specified by the respective inequalities in the $\fam$-mode, which are,
by Corollary \ref{cor.fac-cores}, just those defining SE faces of\/ $\FVP$.
Since both $\FVP$ and $\FVP^{\prime}$ have $\CIP$ as its image by \eqref{eq.eta-to-char},
it follows from Lemma \ref{lem.fac-cores} that the maximization of
$\fam\mapsto \langle\obj ,\fam\rangle_{\Fai} =\langle\zob_{\obj},\c_{\fam}\rangle_{\Cai}$ over
any of them has the same optimal value as the maximization of
$\c\mapsto \langle\zob_{\obj},\c\rangle_{\Cai}$ over $\c$ in the polytope $\CIP$.
\end{proof}

However, most of the inequalities defining SE faces of\/ $\FVP$ are superfluous.
That redundant list can be reduced as follows.

\begin{theorem}\label{thm.the-result}
Let $\obj$ be an SE objective. Then the LP problem to
$$
\mbox{maximize $\fam\mapsto \langle\obj ,\fam\rangle_{\Fai}$ over $\fam\in\FVP$}
$$
has the same optimal value as the LP problem to maximize the same function
over the polyhedron specified by
\begin{itemize}
\item the inequalities defining SE faces that correspond to
those facets of\/ $\CIP$ that do not contain the 0-imset,
\item the non-negativity and modified convexity constraints
\eqref{eq.facet-empty} and \eqref{eq.mod-conv}.
\end{itemize}
\end{theorem}
\smallskip

\begin{proof}
We extend the arguments given in the proof of Lemma \ref{lem.revis-conj}.
The polytope $\CIP$ can be viewed as the polyhedron specified by its facet-defining inequalities.
In particular, the pre-image $\FVP^{\prime}$ of $\CIP$ by  \eqref{eq.eta-to-char} can equivalently
be defined as the polyhedron specified by the facet-defining inequalities for $\CIP$
which are \mbox{re-written} into the $\fam$-mode.

Of course, the conclusion of Lemma \ref{lem.revis-conj} on the same optimal value holds
for any polyhedron $\FVP^{\prime\prime}$ such that
$\FVP\subseteq \FVP^{\prime\prime}\subseteq \FVP^{\prime}$. Thus, in place
of $\FVP^{\prime\prime}$ one can take the polyhedron specified by the corresponding facet-defining inequalities for $\CIP$, the non-negativity and modified convexity constraints.

The last observation is that the
facet-defining inequalities for $\CIP$ that are tight for the 0-imset are implied
by the non-negativity constraints. Indeed, the $\fam$-versions of such inequalities
are tight at the empty graph and the observation follows from Lemma~\ref{lem.increase-facet}(i)-(ii).
Therefore, they can be dropped from the specification of $\FVP^{\prime\prime}$.
\end{proof}

Thus, our aim, when maximizing an SE objective, to eliminate non-SE facets of\/ $\FVP$
except for \eqref{eq.facet-empty} and \eqref{eq.mod-conv} seems to be achieved.
The price for it is that one has to include inequalities that are not
facet-defining for $\FVP$, namely some of the facet-defining inequalities for $\CIP$
written in the $\fam$-mode.

\begin{remark}\rm
It follows from the proof of Theorem \ref{thm.the-result} that the modified convexity
constraints \eqref{eq.mod-conv} are superfluous there. However,
Theorem \ref{thm.the-result} can be strengthened using a stronger
result from \cite{SH13} which says that one can exclude the so-called
{\em specific inequalities} from the list of the inequalities given by SE objectives.
These specific inequalities are shown in \cite[\S\,4.1]{SH13} to be exact translations
of the constraints \eqref{eq.facet-empty} and \eqref{eq.mod-conv} into the
frame of the characteristic imsets. The list of specific inequalities in case $|N|=4$
is given in \S\,\ref{sec.CIP-four}; only one of them, namely \eqref{eq.c20}, needs
\eqref{eq.mod-conv} for its derivation. Thus, in that strengthening of
Theorem \ref{thm.the-result}, the modified convexity inequality \eqref{eq.mod-conv}
must be included.
\end{remark}

It turns out that, in the case of $|N|=4$ the original conjecture---that
SE {\em facets} plus \eqref{eq.facet-empty} and \eqref{eq.mod-conv} are sufficient---is true.

\begin{corollary}\rm\label{cor.the-result-4BN}
If $|N|=4$ then the LP problem to maximize $\fam\mapsto \langle\obj ,\fam\rangle_{\Fai}$
over $\fam\in\FVP$ with an SE objective  $\obj\in {\dv R}^{\Fai}$ has the same optimal value as the LP problem
to maximize the same objective over the polyhedron specified by \eqref{eq.facet-empty} and \eqref{eq.mod-conv}
and the inequalities defining SE facets.
\end{corollary}

\begin{proof}
Use Theorem \ref{thm.the-result}; as shown in Example \ref{exa.four-CIP}, the only
facets of $\CIP$ not implied solely by \eqref{eq.facet-empty}
are those in \eqref{eq.c20}, implied by the combination of \eqref{eq.facet-empty} and \eqref{eq.mod-conv}.
\end{proof}

\section{A counter-example to the original conjecture}\label{sec.five-var}

Recently, Orlinskaya, in her thesis \cite{Orl14} disproved Conjecture 1 from \cite{SV11}
by finding a new facet-defining inequality for the characteristic-imset polytope
$\CIP$ in case $N=\{ a,b,c,d,e\}$, which is neither tight for the 1-imset nor be
one of the earlier-mentioned {\em specific inequalities}, whose $\fam$-versions
are derivable from \eqref{eq.facet-empty} and \eqref{eq.mod-conv}. The inequality has this form:
\begin{eqnarray}
-\,\c (ab) +2\cdot\c (ac) +3\cdot\c (ae) +\c (bc) -\c (bd) +2\cdot\c (cd)  &&\nonumber\\
+\,5\cdot\c (ce) +3\cdot\c (de)  +2\cdot\c (abc)  +4\cdot\c (abd) +3\cdot\c (abe)  && \label{eq.five-var}\\
+\,\c (acd) -2\cdot\c (ace) +2\cdot\c (bcd) -\c (bce) - 3\cdot\c (cde)  -5\cdot\c (abcd)  && \nonumber\\
-\,2\cdot\c (abce) -3\cdot\c (abde) -\c (acde) +\c (bcde) +5\cdot\c (abcde) &\leq& 16\,.\nonumber
\end{eqnarray}
The substitution of \eqref{eq.eta-to-char} gives the family-variable version of the inequality:
\begin{eqnarray}
-\,\fam (a\sepi b) +2\cdot\fam (a\sepi c) +3\cdot\fam (a\sepi e) +3\cdot\fam (a\sepi bc)
+3\cdot\fam (a\sepi bd) && \nonumber \\
+\,5\cdot\fam (a\sepi be) +3\cdot\fam (a\sepi cd) +3\cdot\fam (a\sepi ce)
+3\cdot\fam (a\sepi de)   +3\cdot\fam (a\sepi bcd) && \nonumber \\
+\,5\cdot\fam (a\sepi bce) +6\cdot\fam (a\sepi bde) +3\cdot\fam (a\sepi cde)
+6\cdot\fam (a\sepi bcde)    && \nonumber \\
-\,\fam (b\sepi a) +\fam (b\sepi c) -\fam (b\sepi d)  +2\cdot\fam (b\sepi ac)
+2\cdot\fam (b\sepi ad)&& \nonumber \\
+\,2\cdot\fam (b\sepi ae) +2\cdot\fam (b\sepi cd) -\fam (b\sepi de)  +2\cdot\fam (b\sepi acd)
+2\cdot\fam (b\sepi ace)&& \nonumber \\
+\,2\cdot\fam (b\sepi ade) +2\cdot\fam (b\sepi cde) +5\cdot\fam (b\sepi acde)
+2\cdot\fam (c\sepi a) && \nonumber \\
+\,\fam (c\sepi b) +2\cdot\fam (c\sepi d) +5\cdot\fam (c\sepi e)  +5\cdot\fam (c\sepi ab)
+5\cdot\fam (c\sepi ad) && \label{eq.five-fam}\\
+5\cdot\fam (c\sepi ae) +5\cdot\fam (c\sepi bd) +5\cdot\fam (c\sepi be) +4\cdot\fam (c\sepi de)
+5\cdot\fam (c\sepi abd)&& \nonumber \\
+\,5\cdot\fam (c\sepi abe) +4\cdot\fam (c\sepi ade) +7\cdot\fam (c\sepi bde)
+7\cdot\fam (c\sepi abde)  -\fam (d\sepi b) & \nonumber \\
+\,2\cdot\fam (d\sepi c) +3\cdot\fam (d\sepi e)  +3\cdot\fam (d\sepi ab) +3\cdot\fam (d\sepi ac)
+3\cdot\fam (d\sepi ae)&& \nonumber \\
+\,3\cdot\fam (d\sepi bc) +2\cdot\fam (d\sepi be) +2\cdot\fam (d\sepi ce) +3\cdot\fam (d\sepi abc)
+3\cdot\fam (d\sepi abe) && \nonumber \\
+\,2\cdot\fam (d\sepi ace) +4\cdot\fam (d\sepi bce) +5\cdot\fam (d\sepi abce)  +3\cdot\fam (e\sepi a) && \nonumber \\
+\,5\cdot\fam (e\sepi c) +3\cdot\fam (e\sepi d) +6\cdot\fam (e\sepi ab) +6\cdot\fam (e\sepi ac)
+6\cdot\fam (e\sepi ad)&& \nonumber \\
+\,4\cdot\fam (e\sepi bc) +3\cdot\fam (e\sepi bd) +5\cdot\fam (e\sepi cd)  +6\cdot\fam (e\sepi abc)&& \nonumber \\
+\,6\cdot\fam (e\sepi abd) +5\cdot\fam (e\sepi acd) +5\cdot\fam (e\sepi bcd)
+8\cdot\fam (e\sepi abcd) &\leq& 16\,. \nonumber
\end{eqnarray}
Consider the corresponding SE objective $\obj^{*}\in {\dv R}^{\Fai}$,
that is, for any $(a\sepi B)\in\Fai$, $\obj^{*} (a\sepi B)$ is the coefficient with
$\fam (a\sepi B)$ in \eqref{eq.five-fam}. It follows immediately from Lemma \ref{lem.increase-facet}(iv)
that \eqref{eq.five-fam} is {\em not facet-defining for $\FVP$} because some
coefficients are negative. In fact,the respective SE face of\/ $\FVP$ given by
$\langle\obj^{*} ,\fam\rangle_{\Fai}=16$, denoted below by $F_{*}$, has the dimension 53, which is far from 74, the dimension of facets of $\FVP$.
We checked this fact by means of a computer: we found all 153 codes of acyclic directed
graphs on $F_{*}$; at most 54 of them are affinely independent.
\smallskip

On the other hand, the inequality \eqref{eq.five-var} {\em is facet-defining for $\CIP$}.
We have computed 59 characteristic imsets on this face of $\CIP$, denoted below by $\barF_{*}$, and found
26 of them affinely independent. This implies the dimension of $\barF_{*}$ is 25,
which is the dimension of facets of $\CIP$.
\smallskip

To get the desired counter-example we consider a convex combination $\fam_{\dag}$
of all $153$ codes of acyclic directed graphs on $F_{*}$ with the coefficients $\frac{1}{153}$:
\begin{eqnarray*}
\fam_{\dag} &:=& ~~\frac{4}{153}\cdot\Id_{a\,\separ \{ c \}} +\frac{4}{153}\cdot\Id_{a\,\separ \{ d \}}
+\frac{14}{153}\cdot\Id_{a\,\separ \{ e \}} +\frac{1}{153}\cdot\Id_{a\,\separ \{ b,c \}} \\
&& +\,\frac{10}{153}\cdot\Id_{a\,\separ \{ b,d \}} +\frac{3}{153}\cdot\Id_{a\,\separ \{ b,e \}}
+\frac{8}{153}\cdot\Id_{a\,\separ \{ c,d \}} +\frac{7}{153}\cdot\Id_{a\,\separ \{ c,e \}} \\
&& +\,\frac{7}{153}\cdot\Id_{a\,\separ \{ d,e \}} +\frac{1}{153}\cdot\Id_{a\,\separ \{ b,c,d \}}
+\frac{3}{153}\cdot\Id_{a\,\separ \{ b,c,e \}} +\frac{24}{153}\cdot\Id_{a\,\separ \{ b,d,e \}} \\
&& +\,\frac{3}{153}\cdot\Id_{a\,\separ \{ c,d,e \}} +\frac{18}{153}\cdot\Id_{a\,\separ \{ b,c,d,e \}}
+\frac{8}{153}\cdot\Id_{b\,\separ \{ c \}} +\frac{6}{153}\cdot\Id_{b\,\separ \{ e \}} \\
&& +\,\frac{6}{153}\cdot\Id_{b\,\separ \{ c,d \}} +\frac{6}{153}\cdot\Id_{b\,\separ \{ c,d,e \}}
+\frac{66}{153}\cdot\Id_{b\,\separ \{ a,c,d,e \}} +\frac{4}{153}\cdot\Id_{c\,\separ \{ a \}} \\
&& +\,\frac{4}{153}\cdot\Id_{c\,\separ \{ b \}} +\frac{2}{153}\cdot\Id_{c\,\separ \{ d \}}
+\frac{33}{153}\cdot\Id_{c\,\separ \{ e \}} +\frac{8}{153}\cdot\Id_{c\,\separ \{ a,b \}} \\
&& +\,\frac{15}{153}\cdot\Id_{c\,\separ \{ a,d \}} +\frac{13}{153}\cdot\Id_{c\,\separ \{ a,e \}}
+\frac{11}{153}\cdot\Id_{c\,\separ \{ b,d \}} +\frac{1}{153}\cdot\Id_{c\,\separ \{ b,e \}} \\
&& +\,\frac{2}{153}\cdot\Id_{c\,\separ \{ a,b,d \}} +\frac{1}{153}\cdot\Id_{c\,\separ \{ a,b,e \}}
+\frac{21}{153}\cdot\Id_{c\,\separ \{ b,d,e \}}  \\
&& +\,\frac{15}{153}\cdot\Id_{c\,\separ \{ a,b,d,e \}} +\frac{4}{153}\cdot\Id_{d\,\separ \{ a \}}
+\frac{2}{153}\cdot\Id_{d\,\separ \{ c \}} +\frac{38}{153}\cdot\Id_{d\,\separ \{ e \}} \\
&& +\,\frac{10}{153}\cdot\Id_{d\,\separ \{ a,b \}} +\frac{12}{153}\cdot\Id_{d\,\separ \{ a,c \}}
+\frac{13}{153}\cdot\Id_{d\,\separ \{ a,e \}} +\frac{1}{153}\cdot\Id_{d\,\separ \{ b,c \}} \\
&& +\,\frac{1}{153}\cdot\Id_{d\,\separ \{ a,b,c \}}  +\frac{2}{153}\cdot\Id_{d\,\separ \{ a,b,e \}}
+\frac{6}{153}\cdot\Id_{d\,\separ \{ b,c,e \}} \\
&& +\,\frac{13}{153}\cdot\Id_{d\,\separ \{ a,b,c,e \}} +\frac{8}{153}\cdot\Id_{e\,\separ \{ a \}} +\frac{3}{153}\cdot\Id_{e\,\separ \{ b \}} +\frac{23}{153}\cdot\Id_{e\,\separ \{ c \}} \\
&& +\,\frac{19}{153}\cdot\Id_{e\,\separ \{ d \}} +\frac{15}{153}\cdot\Id_{e\,\separ \{ a,b \}}
+\frac{12}{153}\cdot\Id_{e\,\separ \{ a,c \}} +\frac{17}{153}\cdot\Id_{e\,\separ \{ a,d \}} \\
&& +\,\frac{3}{153}\cdot\Id_{e\,\separ \{ b,d \}} +\frac{2}{153}\cdot\Id_{e\,\separ \{ c,d \}} +\frac{1}{153}\cdot\Id_{e\,\separ \{ a,b,c \}} +\frac{4}{153}\cdot\Id_{e\,\separ \{ a,b,d \}} \\
&& +\,\frac{2}{153}\cdot\Id_{e\,\separ \{ b,c,d \}} +\frac{14}{153}\cdot\Id_{e\,\separ \{ a,b,c,d \}}\,.
\end{eqnarray*}
It is tedious but straightforward to verify $\langle\obj^{*} ,\fam_{\dag}\rangle =16$.
One can also easily check that none of five modified convexity constraints is tight for $\fam_{\dag}$.
We also verified that the vector $\c_{\dag}\in {\dv R}^{\Cai}$ ascribed to $\fam_{\dag}$ by
\eqref{eq.eta-to-char} is in the relative interior of $\barF_{*}\subseteq\CIP$.
For this purpose, we have first used 59 vertices of $\barF_{*}$ to compute its 55 facets.
Then we verified computationally that $\c_{\dag}$ does not belong to any of the 55 facets of $\barF_{*}$.
\smallskip

The above observation implies that none of the SE-facets of\/ $\FVP$ contains $\fam_{\dag}$.
Indeed, assume for a contradiction that $\fam_{\dag}$ belongs to some SE-facet $F$ of\/ $\FVP$.
Then, by Lemma \ref{lem.fac-cores} and Corollary \ref{cor.fac-cores}, $\c_{\dag}$ belongs
to the corresponding face $\barF$ of $\CIP$, which is, by Corollary \ref{cor.SE-facet-cores},
a facet of $\CIP$.  The facet $\barF$ does not contain fully $\barF_{*}$
since otherwise, by Lemma \ref{lem.facet=maxim} applied to $\CIP$ and $\barF_{*}$, one has
$\barF=\barF_{*}$ and, by Corollary \ref{cor.fac-cores}, $F=F_{*}$, contradicting the
above mentioned fact that \eqref{eq.five-fam} is not facet-defining for $\FVP$. Therefore,
$\c_{\dag}\in\barF_{*}\cap\barF\subset\barF_{*}$, which is a contradiction with belonging $\c_{\dag}$ to the relative interior of $\barF_{*}$.
\smallskip

These observations are enough to derive the existence of a counter-example, which is the vector $\fam_{\star}:=(1+\epsilon)\cdot\fam_{\dag}$, where $\epsilon>0$ is small enough.
Indeed, $\fam_{\star}$ satisfies all non-negativity constraints and all other inequalities, namely
5 modified convexity constraints and SE-facets of\/ $\FVP$ are valid for $\fam_{\dag}$ but not tight
{for it: $\langle\obj ,\fam_{\dag}\rangle <u$ for the respective $\obj\in {\dv R}^{\Fai}$ and $u>0$.
Since the number of these inequalities is finite, a small $\epsilon$-perturbation
retains $\langle\obj ,\fam_{\star}\rangle <u$ for any of them.
On the other hand, the value of the considered SE objective
$\obj^{*}\in {\dv R}^{\Fai}$ for $\fam_{\star}$ is
$$
\langle\obj^{*} ,\fam_{\star}\rangle =(1+\epsilon)\cdot\langle\obj^{*} ,\fam_{\dag}\rangle=
(1+\epsilon)\cdot 16>16\,.
$$
Thus, the maximum of the linear SE objective $\fam\mapsto\langle\obj^{*} ,\fam\rangle$,
$\fam\in {\dv R}^{\Fai}$ on $\FVP$ is $16$, while its value in $\fam_{\star}$, which
satisfies \eqref{eq.facet-empty}, \eqref{eq.mod-conv} and all SE facet-defining inequalities for $\FVP$,
exceeds $16$. This gives the desired counter-example.

\section{Conclusions}\label{sec.conclusion}
Let us summarize the main achievements of the paper.
We dealt with two distinguished polytopes used in the ILP approach to BN structure learning,
namely with the {\em family-variable polytope} and the {\em characteristic-imset polytope}.
Being motivated by a common form of linear objectives to be maximized in the BN structure learning
we introduced the concept of a {\em score equivalent (SE) face} of the family-variable
polytope. We further characterized the linear space of the corresponding SE objectives (Lemma \ref{lem.SEF}).

A correspondence has been established between SE faces of the family-variable polytope $\FVP$ and
the faces of the characteristic-imset polytope $\CIP$, which preserves the inclusion of faces
(Corollary \ref{cor.fac-cores}). We observed that SE facets of $\FVP$ correspond to those facets
of $\CIP$ which contain a distinguished vector, called the 1-imset (Corollary \ref{cor.SE-facet-cores}).
These facets were shown to correspond to {\em extreme supermodular functions},
which gives an elegant method to verify that an inequality is SE-facet-defining for $\FVP$
(combined Theorems \ref{thm.WE=SEfacet} and \ref{thm.SEfacet=extsupermo}). To illustrate the method
we showed that the well-known (generalized) cluster inequalities are facet-defining for $\FVP$
(Corollary \ref{cor.cluster}) and derived their form in the context of the characteristic-imset polytope
(Lemma \ref{lem.gen-clust-char}). The correspondence with extreme supermodular set functions (Theorem~\ref{thm.SEfacet=extsupermo}) may appear to be useful because of a recent extremality criterion for supermodular functions from \cite{SK14}.

Since a typical linear objective appearing in the ILP approach to learning BN structure is special,
namely SE, we raised the question whether all facets of $\FVP$ are needed to specify the feasible sets
for (integer) linear programs when such an objective is maximized. We succeeded in showing that one can eliminate
those facets of\/ $\FVP$ that are not SE, that is, defined by a non-SE normal vectors
(Lemma~\ref{lem.revis-conj}, Theorem~\ref{thm.the-result}).
Nevertheless, our starting original conjecture that one can,
besides simple non-negativity and modified convexity constraints, limit oneself only to
SE facets of\/ $\FVP$ turned out not to be true (a counter-example is given in \S\,\ref{sec.five-var}).
The moral is that one has to consider the inequalities defining facets of the
characteristic-imset polytope $\CIP$ despite the fact that they do not define facets in the context
of the family-variable polytope $\FVP$.

This leads to a suggestion to use a combined coding of BN structures in the ILP approach.
One can encode a BN structure by a concatenation of the family-variable vector and the characteristic imset
and utilize the linear relation \eqref{eq.eta-to-char}. Linear constraints tight at the empty graph are better
represented by simple non-negativity and modified convexity inequalities
in the family-variable part, while the other SE linear inequality constraints
can be more naturally represented in the characteristic-imset part.

We left some of the questions open. One of them is whether a simple condition of
being closed under Markov equivalence characterizes the sets of graph-codes
belonging to SE faces of $\FVP$ (Conjecture \ref{conj.SE=WEface}). However,
it looks like the answer to this question is not essential for the practical application
of ILP methods in BN structure learning.

\begin{acknowledgements}
The research of Milan Studen\'{y} has been supported by the grant GA\v{C}R n.\
13-20012S. James Cussens was supported by the UK Medical Research
Council, grant G1002312 and senior postdoctoral fellowship SF/14/008 from KU Leuven.
Our special thanks are devoted to Fero Mat\'{u}\v{s}, who helped us to find an easy proof of the combinatorial identity from Lemma \ref{lem.combin}.
\end{acknowledgements}



\appendix
\normalsize

\section{Combinatorial identity}\label{sec.comb-iden}

\begin{lemma}\label{lem.combin}\rm
For every non-negative integers $s\geq 0$, $k\geq K\geq 0$ one has
\begin{equation}
\sum_{m=0}^{s} {(-1)}^{m}\cdot {k+s\choose k+m}\cdot {m+k-K\choose m} = {s+K-1 \choose K-1}\,,
\label{eq.combi}
\end{equation}
with conventions ${n\choose 0}={n\choose n}=1$ for any $n\in {\dv Z}$
and ${n\choose -1}={n\choose n+1}=0$ for any non-negative $n\in {\dv Z}$. In particular,
\begin{equation}
\forall\, s\geq 0,\ k\geq 1 ~~\mbox{integers}\qquad
\sum_{m=0}^{s} {(-1)}^{m}\cdot {k+s\choose k+m}\cdot {m+k-1\choose m} = 1\,.
\label{eq.iden-1}
\end{equation}
\end{lemma}

\begin{proof}
The proof relies on the Pascal's triangle identity
$$
{n\choose r}={n-1\choose r}+{n-1\choose r-1}\quad
\mbox{valid for integers $n\geq 1$, $n\geq r\geq 0$.}
$$
Let us denote the sum in \eqref{eq.combi} by $\Sigma (s,k,K)$;
the basic idea of the proof is the induction on $s+K$. First, we verify \eqref{eq.combi} in the case $s=0$:
$$
\Sigma (s=0,k,K)={(-1)}^{0}\cdot {k+0\choose k+0}\cdot {k-K\choose 0}=1=
{0+K-1\choose K-1}\,.
$$
Further special easy case is $s\geq 1$ and $K=0$, in which case
\begin{eqnarray*}
\lefteqn{\Sigma (s\geq 1,k,K=0) = \sum_{m=0}^{s} {(-1)}^{m}\cdot {k+s\choose k+m}\cdot {m+k\choose m}}\\
&=& \sum_{m=0}^{s} {(-1)}^{m}\cdot \frac{(k+s)!}{(k+m)!\cdot (s-m)!} \cdot \frac{(k+m)!}{m!\cdot k!}\\
&=& \frac{(k+s)!}{k!\cdot s!} \cdot \sum_{m=0}^{s} {(-1)}^{m}\cdot \frac{s!}{m!\cdot (s-m)!}
= {k+s\choose k}  \cdot \sum_{m=0}^{s} {(-1)}^{m}\cdot {s\choose m} \\
&=&  {k+s\choose k}  \cdot {(-1+1)}^{s}= 0=
{s-1\choose -1}\,.
\end{eqnarray*}
Thus, \eqref{eq.combi} holds in cases $s=0$ and $K=0$; in particular, if $s+K\leq 1$.
In case $s,K\geq 1$ the induction premise means \eqref{eq.combi} holds for $s^{\prime},K^{\prime}\geq 0$ with
$s^{\prime}+K^{\prime}\leq s+K-1$. To verify the induction step write by the identity
$$
{k+s\choose k+m}={k+s-1\choose k+m}+{k+s-1\choose k+m-1},\quad \mbox{use}~~
{k+s-1\choose k+s}=0\,,
$$
apply the induction premise and use the Pascal's triangle identity again:
\begin{eqnarray*}
\Sigma (s,k,K) &=& \sum_{m=0}^{s} {(-1)}^{m}\cdot {k+s\choose k+m}\cdot {m+k-K\choose m} ~~~~~~~~~~\\
&=& \sum_{m=0}^{s-1} {(-1)}^{m}\cdot {k+s-1\choose k+m}\cdot {m+k-K\choose m} \\
&& ~~~+
\sum_{m=0}^{s} {(-1)}^{m}\cdot {k+s-1\choose k+m-1}\cdot {m+k-K\choose m}\\
&=& \Sigma (s-1,k,K)+\Sigma (s,k-1,K-1) \\
&=& {s+K-2\choose K-1} +{s+K-2\choose K-2} = {s+K-1\choose K-1}\,,
\end{eqnarray*}
which gives the desired result. Putting $K=1$ gives \eqref{eq.iden-1}.
\end{proof}

\section{SE facets in case of four BN variables}\label{sec.SE-four}
There exist 37 SE facets of\/ $\FVP$ in the case $N=\{ a,b,c,d\}$ which decompose
into 10 permutations types. Below we list all the types of the inequalities, both in the family-variable mode
and in the characteristic-imset mode. The generalized cluster inequalities are indicated by $\bullet$,
the remaining types by $\circ$; those are also labeled by the notation used in the
catalogue from \cite{CJKB15}.
\smallskip

\begin{itemize}
\item[$\bullet$] the (generalized) cluster inequality for $C=\{ a,b\}$ (and $k=1$),
\small
$$
[\,\fam (a\sepi b)+\fam (a\sepi bc)+\fam (a\sepi bd)+\fam (a\sepi bcd)\,]
+[\,\fam (b\sepi a)+\fam (b\sepi ac)+\fam (b\sepi ad)+\fam (b\sepi acd)\,]\leq 1\,,
$$
\normalsize
(6 inequalities of this type),~ in characteristic imsets
$$
\c (ab)\leq 1\,,
$$
\item[$\bullet$] the generalized cluster inequality for $C=\{ a,b,c\}$ and $k=2$,
\small
$$
[\,\fam (a\sepi bc)+\fam (a\sepi bcd)\,] +[\,\fam (b\sepi ac)+\fam (b\sepi acd)\,]
+[\,\fam (c\sepi ab)+\fam (c\sepi abd)\,]\leq 1\,,
$$
\normalsize
(4 inequalities of this type),~ in characteristic imsets
$$
\c (abc)\leq 1\,,
$$
\item[$\bullet$] the (generalized) cluster inequality for $C=\{ a,b,c\}$ (and $k=1$),
\small
\begin{eqnarray*}
[\,\fam (a\sepi b)+\fam (a\sepi c)+\fam (a\sepi bc)+\fam (a\sepi bd)+\fam (a\sepi cd)+\fam (a\sepi bcd)\,] &&\\
+[\,\fam (b\sepi a)+\fam (b\sepi c)+\fam (b\sepi ac)+\fam (b\sepi ad)+\fam (b\sepi cd)+\fam (b\sepi acd)\,]
&&\\
+[\,\fam (c\sepi a)+\fam (c\sepi b)+\fam (c\sepi ab)+\fam (c\sepi ad)+\fam (c\sepi bd)+\fam (c\sepi abd)\,]
&\leq & 2\,,
\end{eqnarray*}
\normalsize
(4 inequalities of this type),~  in characteristic imsets
$$
\c (ab)+\c (ac) +\c(bc) - \c (abc) \leq 2\,,
$$
\item[$\bullet$] the generalized cluster inequality for $C=\{ a,b,c,d\}$ and $k=3$,
\small
$$
[\,\fam (a\sepi bcd)+\fam (b\sepi acd)+\fam (c\sepi abd)+\fam (d\sepi abc)\,]\leq 1\,,
$$
\normalsize
(1 inequality of this type),~ in characteristic imsets
$$
\c (abcd)\leq 1\,,
$$
\item[$\bullet$] the generalized cluster inequality for $C=\{ a,b,c,d\}$ and $k=2$,
\small
\begin{eqnarray*}
[\,\fam (a\sepi bc)+\fam (a\sepi bd)+\fam (a\sepi cd)+\fam (a\sepi bcd)\,] &&\\
+[\,\fam (b\sepi ac)+\fam (b\sepi ad)+\fam (b\sepi cd)+\fam (b\sepi acd)\,] &&\\
+[\,\fam (c\sepi ab)+\fam (c\sepi ad)+\fam (c\sepi bd)+\fam (c\sepi abd)\,] &&\\
+[\,\fam (d\sepi ab)+\fam (d\sepi ac)+\fam (d\sepi bc)+\fam (d\sepi abc)\,]  &\leq & 2\,,
\end{eqnarray*}
\normalsize
(1 inequality of this type),~ in characteristic imsets
$$
\c (abc) + \c (abd) +\c (acd) +\c(bcd) - 2\cdot\c (abcd) \leq 2\,,
$$
\item[$\bullet$] the (generalized) cluster inequality for $C=\{ a,b,c,d\}$ (and $k=1$),
\small
\begin{eqnarray*}
[\,\fam (a\sepi b)+\fam (a\sepi c)+\fam (a\sepi d)
+\fam (a\sepi bc)+\fam (a\sepi bd)+\fam (a\sepi cd)+\fam (a\sepi bcd)\,] &&\\
+[\,\fam (b\sepi a)+\fam (b\sepi c)+\fam (b\sepi d)
+\fam (b\sepi ac)+\fam (b\sepi ad)+\fam (b\sepi cd)+\fam (b\sepi acd)\,] &&\\
+[\,\fam (c\sepi a)+\fam (c\sepi b)+\fam (c\sepi d)
+\fam (c\sepi ab)+\fam (c\sepi ad)+\fam (c\sepi bd)+\fam (c\sepi abd)\,] &&\\
+[\,\fam (d\sepi a)+\fam (d\sepi b)+\fam (d\sepi c)
+\fam (d\sepi ab)+\fam (d\sepi ac)+\fam (d\sepi bc)+\fam (d\sepi abc)\,]  &\leq & 3\,,
\end{eqnarray*}
\normalsize
(1 inequality of this type),~ in characteristic imsets
\begin{eqnarray*}
\c (ab) + \c (ac) +\c (ad) +\c (bc) +\c(bd) +\c (cd)\hspace*{6mm}&&\\
-\c (abc) -\c (abd)
-\c (acd) -\c (bcd) +\c (abcd) &\leq& 3\,,
\end{eqnarray*}
\item[$\circ$]  non-cluster SE inequality with 13 terms\\[0.5ex]
[a-a,b-2-acd,c-2-abd,d-2-abc FACETS $a|bcd$]
\small
\begin{eqnarray*}
[\,\fam (a\sepi bc)+\fam (a\sepi bd)+\fam (a\sepi cd)+2\cdot\fam (a\sepi bcd)\,] &&\\
+[\,\fam (b\sepi ac)+\fam (b\sepi ad)+\fam (b\sepi acd)\,] &&\\
+[\,\fam (c\sepi ab)+\fam (c\sepi ad)+\fam (c\sepi abd)\,] &&\\
+[\,\fam (d\sepi ab)+\fam (d\sepi ac)+\fam (d\sepi abc)\,]  &\leq & 2\,,
\end{eqnarray*}
\normalsize
(4 inequalities of this type),~  in characteristic imsets
$$
\c (abc)+\c (abd) +\c(acd) - \c (abcd) \leq 2\,,
$$
\item[$\circ$] non-cluster SE inequality with 16 terms\\[0.5ex]
[a-2-bcd,b-2-acd,c-ab,d-ab FACETS $ab|cd$]
\small
\begin{eqnarray*}
[\,\fam (a\sepi b)+\fam (a\sepi bc)+\fam (a\sepi bd)+\fam (a\sepi cd)+\fam (a\sepi bcd)\,] &&\\
+[\,\fam (b\sepi a)+\fam (b\sepi ac)+\fam (b\sepi ad)+\fam (b\sepi cd)+\fam (b\sepi acd)\,]&&\\
+[\,\fam (c\sepi ad)+\fam (c\sepi bd)+\fam (c\sepi abd)\,]&&\\
+[\,\fam (d\sepi ac)+\fam (d\sepi bc)+\fam (d\sepi abc)\,]  &\leq & 2\,,
\end{eqnarray*}
\normalsize
(6 inequalities of this type),~  in characteristic imsets
$$
\c (ab)+\c (acd) +\c(bcd) - \c (abcd) \leq 2\,,
$$
\item[$\circ$] non-cluster SE inequality with 22 terms\\[0.5ex]
[a-a,a-2-bcd,b-ac,b-ad,c-ab,c-ad,d-ab,d-ac FACETS $a|bcd$]
\small
\begin{eqnarray*}
[\,\fam (a\sepi b)+\fam (a\sepi c)+\fam (a\sepi d)
+2\cdot\fam (a\sepi bc)+2\cdot\fam (a\sepi bd)
+2\cdot\fam (a\sepi cd)+2\cdot\fam (a\sepi bcd)\,] &&\\
+[\,\fam (b\sepi a)+\fam (b\sepi ac)+\fam (b\sepi ad)+\fam (b\sepi cd)+\fam (b\sepi acd)\,] &&\\
+[\,\fam (c\sepi a)+\fam (c\sepi ab)+\fam (c\sepi ad)+\fam (c\sepi bd)+\fam (c\sepi abd)\,] &&\\
+[\,\fam (d\sepi a)+\fam (d\sepi ab)+\fam (d\sepi ac)+\fam (d\sepi bc)+\fam (d\sepi abc)\,]  &\leq &3,
\end{eqnarray*}
\normalsize
(4 inequalities of this type),~  in characteristic imsets
$$
\c (ab) + \c (ac) +\c (ad) +\c(bcd) - \c (abcd) \leq 3\,,
$$
\item[$\circ$] non-cluster SE inequality with 26 terms\\[0.5ex]
[a-a,a-bcd,b-b,b-acd,c-c,c-ad,c-bd,d-d,d-ac,d-bc FACETS $ab|cd$]
\small
\begin{eqnarray*}
[\,\fam (a\sepi b)+\fam (a\sepi c)+\fam (a\sepi d)
+\fam (a\sepi bc)+\fam (a\sepi bd)+2\cdot\fam (a\sepi cd)+2\cdot\fam (a\sepi bcd)\,] &&\\
+[\,\fam (b\sepi a)+\fam (b\sepi c)+\fam (b\sepi d)
+\fam (b\sepi ac)+\fam (b\sepi ad)+2\cdot\fam (b\sepi cd)+2\cdot\fam (b\sepi acd)\,] &&\\
+[\,\fam (c\sepi a)+\fam (c\sepi b)
+\fam (c\sepi ab)+\fam (c\sepi ad)+\fam (c\sepi bd)+2\cdot\fam (c\sepi abd)\,] &&\\
+[\,\fam (d\sepi a)+\fam (d\sepi b)
+\fam (d\sepi ab)+\fam (d\sepi ac)+\fam (d\sepi bc)+2\cdot\fam (d\sepi abc)\,]  &\leq &4\,,
\end{eqnarray*}
\normalsize
(6 inequalities of this type),~  in characteristic imsets
$$
\c (ab) + \c (ac) +\c (ad) +\c (bc) +\c(bd) - \c (abc) -\c (abd) +\c (abcd) \leq 4\,.
$$
\end{itemize}

\section{Specific inequalities in case of four BN variables}\label{sec.CIP-four}
In the case $|N|=4$, the characteristic-imset polytope $\CIP$ has,
besides 37 facets containing the 1-imset and listed in \S\,\ref{sec.SE-four}, additional 117 specific facets that do not contain the 1-imset. They all are defined by means of the so-called
{\em specific inequalities\/} discussed in \cite[\S\,4.1.2]{SH13}. Each of these inequalities
corresponds to a clutter (= Sperner family) of non-empty subsets of $N$, that is,
to a class of inclusion-incomparable subsets of $N$.
The 117 specific facets decompose into 20 permutation types listed below.
Except 4 facets belonging to the last type, mentioned earlier in \eqref{eq.c20}, all of them
contain the 0-imset.
\smallskip

\begin{itemize}
\item[$\circ$] $-\c (ab)\leq 0$~ (6 inequalities of this type),\\[0.5ex]
Sperner family is\, $\calI =\{ ab\}$,\smallskip
\item[$\circ$] $-\c (abc)\leq 0$~ (4 inequalities of this type),\\[0.5ex]
Sperner family is\, $\calI =\{ abc\}$,\smallskip
\item[$\circ$]
$-\c (abcd)\leq 0$~ (1 inequality of this type),\\[0.5ex]
Sperner family is\, $\calI =\{ abcd\}$,\smallskip
\item[$\circ$]
$-\c (ab)-\c (ac)-\c(bc)+2\cdot\c (abc) \leq 0$~ (4 inequalities of this type),\\[0.5ex]
Sperner family is\, $\calI =\{ ab, ac, bc\}$,\smallskip
\item[$\circ$]
$-\c (ab)-\c (acd)-\c(bcd)+2\cdot\c (abcd) \leq 0$~ (6 inequalities of this type),\\[0.5ex]
Sperner family is\, $\calI =\{ ab, acd, bcd\}$,\smallskip
\item[$\circ$]
$-\c (abc)-\c (abd)-\c(acd)+2\cdot\c (abcd)\leq 0$~ (4 inequalities of this type),\\[0.5ex]
Sperner family is\, $\calI =\{ abc, abd, acd\}$,\smallskip
\item[$\circ$]
$-\c (abc)-\c (abd)-\c (acd)-\c(bcd)+3\cdot\c (abcd)\leq 0$~ (1 inequality),\\[0.5ex]
Sperner family is\, $\calI =\{ abc, abd, acd, bcd\}$,\smallskip
\item[$\circ$]
$-\c (ab)-\c (ac)-\c (ad)-\c(bcd)+\c (abc)+\c (abd)+\c (acd)\leq 0$~ (4 inequalities),\\[0.5ex]
Sperner family is\, $\calI =\{ ab, ac, ad, bcd\}$,\smallskip
\item[$\circ$]
$-\c (ab)-\c (ac)-\c (ad)-\c (bc)-\c(bd)$\\[0.4ex]
\hspace*{2mm}$+2\cdot\c (abc)+2\cdot\c (abd)
+\c (acd)+\c (bcd)-2\cdot\c (abcd)\leq 0$~ (6 inequalities),\\[0.5ex]
Sperner family is\, $\calI =\{ ab, ac, ad, bc, bd\}$,\smallskip
\item[$\circ$]
$-\c (ab)-\c (ac)-\c (ad)-\c (bc)-\c(bd)-\c (cd)$\\[0.4ex]
\hspace*{2mm}$+2\cdot\c (abc)+2\cdot\c (abd)
+2\cdot\c (acd)+2\cdot\c (bcd)-3\cdot\c (abcd)\leq 0$~ (1 inequality),\\[0.5ex]
Sperner family is\, $\calI =\{ ab, ac, ad, bc, bd, cd\}$,\smallskip
\item[$\circ$]
$-\c (ab)-\c (ac)+\c (abc)\leq 0$~ (12 inequalities of this type),\\[0.5ex]
Sperner family is\, $\calI =\{ ab, ac\}$,\smallskip
\item[$\circ$]
$-\c (abc)-\c (abd)+\c (abcd)\leq 0$~ (6 inequalities of this type),\\[0.5ex]
Sperner family is\, $\calI =\{ abc, abd\}$,\smallskip
\item[$\circ$]
$-\c (ab)-\c (acd)+\c (abcd)\leq 0$~ (12 inequalities of this type),\\[0.5ex]
Sperner family is\, $\calI =\{ ab, acd\}$,\smallskip
\item[$\circ$]
$-\c (ab)-\c (cd)+\c (abcd)\leq 0$~ (3 inequalities of this type),\\[0.5ex]
Sperner family is\, $\calI =\{ ab, cd\}$,\smallskip
\item[$\circ$]
$-\c (ab)-\c (ac)-\c (ad)+\c (abc)+\c (abd)+\c (acd)-\c (abcd)\leq 0$~ (4 inequalities),\\[0.5ex]
Sperner family is\, $\calI =\{ ab, ac, ad\}$,\smallskip
\item[$\circ$]
$-\c (ab)-\c (ac)-\c (bd)+\c (abc)+\c (abd)\leq 0$~
(12 inequalities of this type),\\[0.5ex]
Sperner family is\, $\calI =\{ ab, ac, bd\}$,\smallskip
\item[$\circ$]
$-\c (ab)-\c (ac)-\c(bcd)+\c (abc)+\c (abcd)\leq 0$~ (12 inequalities of this type),\\[0.5ex]
Sperner family is\, $\calI =\{ ab, ac, bcd\}$,\smallskip
\item[$\circ$]
$-\c (ab)-\c (ac)-\c (bc)-\c(cd)+2\cdot\c (abc)+\c (acd)+\c (bcd)-\c (abcd)\leq 0$\\[0.4ex]
(12 inequalities of this type),\\[0.5ex]
Sperner family is\, $\calI =\{ ab, ac, bc, cd\}$,\smallskip
\item[$\circ$]
$-\c (ab)-\c (ad)-\c (bc)-\c (cd)$\\[0.4ex]
\hspace*{2mm}$+\c (abc)+\c (abd)
+\c (acd)+\c (bcd)-\c (abcd)\leq 0$~ (3 inequalities),\\[0.5ex]
Sperner family is\, $\calI =\{ ab, ad, bc, cd\}$,\smallskip
\item[$\circ$]
$-\c (bc)-\c(bd)-\c (cd)+\c (abc)+\c (abd)
+\c (acd)+2\cdot\c (bcd)-2\cdot\c (abcd)\leq 1$\\[0.4ex]
(4 inequalities of this type),\\[0.5ex]
Sperner family is\, $\calI =\{ a, bc, bd, cd\}$.
\end{itemize}

\end{document}